\documentclass[a4paper,11pt]{amsart}
\usepackage{amsmath,  amsfonts, amssymb, amsthm, amscd}
\usepackage{graphicx}
\usepackage{enumerate}
\usepackage{float}
\usepackage{fontawesome5}
\usepackage{eso-pic}
\usepackage{charter}
\usepackage{url}
\usepackage{color}
\usepackage[utf8]{inputenc}
\usepackage{tikz}
\usetikzlibrary{arrows}
\usepackage[a4paper]{geometry}
\usepackage[colorlinks=true,linkcolor=blue,citecolor=magenta]{hyperref}
\usepackage{epsfig, enumerate}
\usepackage{graphicx}
\usepackage{enumerate}
\usepackage{booktabs}
\usepackage{dsfont}
\usepackage[utf8]{inputenc}
\usepackage[T1]{fontenc}
\usepackage{mathrsfs}
\usepackage{a4wide}
\usetikzlibrary{calc}
\usetikzlibrary{external}
\numberwithin{equation}{section}

\newtheorem{theorem}{Theorem}[section]
\newtheorem{lemma}[theorem]{Lemma}
\newtheorem{proposition}[theorem]{Proposition}
\newtheorem{corollary}[theorem]{Corollary}

\newtheorem{definition}[theorem]{Definition}

\newcommand{\mc}[1]{{\mathcal #1}}
\newcommand{\mf}[1]{{\mathfrak #1}}

\newcommand{\bb}[1]{{\mathbb #1}}

\newcommand{\<}{\langle}
\renewcommand{\>}{\rangle}

\renewcommand{\epsilon}{\varepsilon}

\def\centerarc[#1](#2)(#3:#4:#5){\draw[#1] ($(#2)+({#5*cos(#3)},{#5*sin(#3)})$) arc (#3:#4:#5);}
\newcommand{\pfrac}[2]{\genfrac{}{}{}{1}{#1}{#2}}

\newcommand{\bola}{\textrm{\Large\textperiodcentered}}

\newcommand\topo[2]{\genfrac{}{}{0pt}{}{#1}{#2}}

\renewcommand{\tilde}{\widetilde}          % wider `tilde'
\DeclareMathSymbol{\leqslant}{\mathalpha}{AMSa}{"36} % nicer `smaller or equal'
\DeclareMathSymbol{\geqslant}{\mathalpha}{AMSa}{"3E} % nicer `larger or equal'
\DeclareMathSymbol{\eset}{\mathalpha}{AMSb}{"3F}     % nicer `emptyset'
\newcommand{\T}{\mathbb{T}}

\renewcommand{\epsilon}{\varepsilon}
\newenvironment{myenumerate}{
	\renewcommand{\theenumi}{\arabic{enumi}}
	\renewcommand{\labelenumi}{{\rm(\theenumi)}}
	\begin{list}{\labelenumi}
		{
			\setlength{\itemsep}{0.4em}
			\setlength{\topsep}{0.5em}
			\setlength\leftmargin{2.45em}
			\setlength\labelwidth{2.05em}
			\setlength{\labelsep}{0.4em}
			\usecounter{enumi}
		}
	}
	{\end{list}
}

\newcommand{\beq}{\begin{equation}}
\newcommand{\eeq}{\end{equation}}
\newcommand{\ba}{\begin{aligned}}
	\newcommand{\ea}{\end{aligned}}

\usetikzlibrary{shapes.misc}
\usetikzlibrary{shapes.symbols}
\usetikzlibrary{shapes.geometric}
\usetikzlibrary{decorations}
\usetikzlibrary{decorations.markings}

\tikzstyle{tinydots}=[dash pattern=on \pgflinewidth off 2*\pgflinewidth]

\let\oldtocsection=\tocsection
\let\oldtocsubsection=\tocsubsection
\let\oldtocsubsubsection=\tocsubsubsection
\renewcommand{\tocsection}[2]{\hspace{0em}\oldtocsection{#1}{#2}}
\renewcommand{\tocsubsection}[2]{\hspace{1em}\oldtocsubsection{#1}{#2}}
\renewcommand{\tocsubsubsection}[2]{\hspace{2em}\oldtocsubsubsection{#1}{#2}}
\DeclareRobustCommand{\SkipTocEntry}[5]{}
\begin{document}

\title[Superdiffusive Scaling Limits for SEP Slow Bonds]{Superdiffusive Scaling Limits for the Symmetric\\ Exclusion Process with Slow Bonds}

\author{Dirk Erhard}
\address{UFBA\\
	Instituto de Matem\'atica, Campus de Ondina, Av. Adhemar de Barros, S/N. CEP 40170-110\\
	Salvador, Brazil}
\curraddr{}
\email{erharddirk@gmail.com}
\thanks{}

\author{Tertuliano Franco}
\address{UFBA\\
	Instituto de Matem\'atica, Campus de Ondina, Av. Adhemar de Barros, S/N. CEP 40170-110\\
	Salvador, Brazil}
\curraddr{}
\email{tertu@ufba.br}
\thanks{}

\author{Tiecheng Xu}
\address{UFBA\\
	Instituto de Matem\'atica, Campus de Ondina, Av. Adhemar de Barros, S/N. CEP 40170-110\\
	Salvador, Brazil}
\thanks{}
\curraddr{}
\email{xutcmath@gmail.com}


\begin{abstract}
	In \cite{fgn1},  the hydrodynamic limit in the diffusive scaling of the symmetric simple exclusion process with a finite number of slow bonds of strength $n^{-\beta}$ has been studied. Here $n$ is the scaling parameter and $\beta>0$ is fixed. As shown in \cite{fgn1}, when $\beta>1$, such a limit is given by the heat equation with Neumann boundary conditions. In this work, we find more non-trivial super-diffusive scaling limits for this dynamics. Assume that there are $k$ equally spaced slow bonds in the system.
	If $k$ is fixed and the time scale is $k^2n^\theta$, with $\theta\in (2,1+\beta)$, the density is asymptotically  constant in each of the $k$ boxes, and equal to the initial expected mass in that box, i.e., there is no time evolution.  If $k$ is fixed and the time scale is $k^2n^{1+\beta}$, then  the density is also spatially constant in each box, but evolves in time according to the discrete heat equation. Finally, if the time scale is $k^2n^{1+\beta}$ and, additionally, the number of boxes $k$ increases to infinity, then the system converges to the  continuous heat equation on the torus, with no boundary conditions.
	\end{abstract}

\maketitle

\tableofcontents

\section{Introduction}\label{s1}
The subject of interacting particle systems involving boundary conditions and  how those boundary dynamics influence its macroscopic limits has attracted a lot of attention in recent decades. As just one example among many, Franco, Gonçalves, and Neumann studied the hydrodynamic limit of symmetric exclusion processes with a slow bond in \cite{fgn1}. The exclusion process is a well-known and widely studied model in Probability and Statistical Mechanics. It involves random walks with a hard-core interaction, meaning that each vertex of the graph can host at most one particle.

In a symmetric exclusion process, the jump rates are associated with the edges of the graph: when the Poisson clock of an edge rings, the particles at the two vertices connected by that edge swap places. If we consider as the underlying graph a discrete torus with
$n$ vertices, the slow bond corresponds to an edge whose jump rate is smaller than of the other bonds. More precisely, while most edges have a jump rate of one, the slow bond has a rate of $n^{-\beta}$, where $\beta\geq 0$ is fixed parameter.

Combining \cite{fgn1} and \cite{fgn2}, it follows that the hydrodynamic limit of the symmetric exclusion process exhibits a phase transition in $\beta$. If $\beta\in [0,1)$, the hydrodynamic limit is driven by the heat equation on the torus (that is, no boundary condition appears in the limit). If $\beta=1$, then the hydrodynamic equation is the heat equation with Robin boundary condition at the macroscopic point close to the slow bond. Moreover,  this Robin boundary condition can be interpreted as Fourier's Law at that point. Finally, if $\beta>1$, then the hydrodynamic equation is given by the heat equation with Neumann boundary conditions, meaning that the slow bond in this case is strong enough to create an impenetrable barrier in the limit.

In this paper, we investigate the case $\beta>1$ in more detail. We show the existence of three more scaling limits for this system. To be more precise, we consider not a single slow bond, but $k$ slow bonds of rate $n^{-\beta}$, $\beta>1$, equally spaced in the discrete torus with $nk$ sites. A \textit{box} refers to the $n$ sites between two consecutive slow bonds.

	If $k$ is fixed and the time scale  is  $k^2n^{2+\theta}$, with $\theta\in (0,\beta-1)$, then  the particle density evolution   is asymptotically  constant in each of the $k$ boxes, and equal to the initial mass in each box between two slow bonds;  that is, the equilibrium  is achieved instantaneously inside each box, and  there is  no sharing of mass between boxes (thus no time evolution).

	 If $k$ is fixed and the time scale is $k^2n^{1+\beta}$, then  the density is also spatially constant in each box, but now the boxes shares masses; furthermore  the vector of densities (of each box) evolves in time according to the discrete heat equation.

	 Finally, if the time scale is $k^2n^{1+\beta}$ and, additionally, the number of boxes $k$ increases to infinity, under a proper definition of the empirical measure, the system converges to the  continuous heat equation on the torus, with no boundary conditions. In other words, each box can be seen as a single site, those sites interact and its hydrodynamic limit is the same as the one of the homogeneous symmetric exclusion process.

	 It is worth mentioning the different approaches required here. For the case where $k$ is fixed, we follow Varadhan's Entropy Method, while for $k \to \infty $, we use the refined Yao's Relative Entropy Method, developed by Jara and Menezes~\cite{jm}. These two approaches were not merely a choice; as will become clear in the paper, when considering the time scale $k^2n^{1+\beta}$, the Replacement Lemma within a box (a fundamental component of the entropy method) only holds when $k $ does not grow too fast, namely, only if $k = o(n^{\frac{\beta-1}{2}}) $. On the other hand, as we will show, the refined Yao's Relative Entropy Method of Jara and Menezes works for any $k \to \infty$, but not when $k $ is fixed.

	 The paper is divided as follows: in Section~\ref{s2} we state results; in Section~\ref{s3} we deal with the scenario where $k$ is fixed by following the Varadhan's Entropy Method approach; in Section~\ref{s4} we deal with with the scenario where $k=k(n)$ converges to infinity by following the refined Yao's Relative Entropy Method of Jara and Menezes.

\section{Statement of results}\label{s2}

Denote by $\bb T = \bb R/\bb Z = [0, 1)$ the one-dimensional continuous torus and by $\bb T_{nk} = \bb Z/nk\bb Z =\{0,\ldots, nk - 1\}$ the one-dimensional discrete torus with $nk$ points,  where $k=k(n)$ will be specified later. Define also the state space $\Omega_{nk}:=\{0,1\}^{\mathbb{T}_{nk}}$.

Fix $\alpha>0$ and $\beta\geq 0$. We consider the \emph{exclusion process with slow bonds} which  is the Markov process $\{\eta^n_t:t\geq{0}\}$ whose infinitesimal generator $\mc{L}_{n}$ acts on local functions $f:\Omega_{nk}\rightarrow{\mathbb{R}}$ as
\begin{equation}\label{ln}
	(\mc{L}_{n}f)(\eta)=\sum_{x\in \bb T_{nk}}\,\xi^{nk}_{x,x+1}\,[f(\eta^{x,x+1})-f(\eta)]\,,
\end{equation}
where
$\eta^{x,x+1}$ is the configuration obtained from $\eta$ by exchanging the occupation variables $\eta(x)$ and $\eta(x+1)$, i.e.,
\begin{equation}\label{exchange}
	(\eta^{x,x+1})(y)=\left\{\begin{array}{cl}
		\eta(x),& \mbox{if}\,\,\, y=x+1\,,\\
		\eta(x+1),& \mbox{if} \,\,\,y=x\,,\\
		\eta(y),& \mbox{otherwise}
	\end{array}
	\right.
\end{equation}
the conductantes $\xi^{nk}_{x,x+1}$ are given by
\begin{equation*}
	\xi^{nk}_{x,x+1}\;=\;
	\begin{cases}
		\alpha n^{-\beta}\,,& \text{ if } x= jn-1 \text{ for } j\in\{1,2,\ldots,k\},\\
		1\,,& \text{ otherwise. }\\
	\end{cases}
\end{equation*}
Throughout this article we shall assume that $\beta>1$, and when we say that the process is \textbf{speeded up by a function} $f(n)$, we mean that we are considering the process under the generator $f(n) \mc L_n$.

Let $\mathcal{M}$ be the compact space of non-negative measures on $\bb T$ with total mass bounded by one endowed with the weak topology. Denote by $\Pi_t^n \in \mc M$ the usual empirical measure on $\bb T$ associated to the configuration $\eta_t$:
\begin{equation*}
	\Pi_t^n(du)\;=\; \Pi^n(\eta_t,du)\;:=\;\frac{1}{nk}\sum_{x\in\bb T_{nk}}\eta_t(x) \delta_{x/nk}(du)\,.
\end{equation*}
 We would like to investigate the scaling limit of $\Pi^n$. The behavior of the system we shall prove is the following: under the time scale $k^2n^{2+\theta}$ for any $\theta>0$, the process reaches equilibrium immediately inside each box, but the behavior exhibits a phase transition at $\theta=\beta-1$. Namely, the average occupations of boxes stay unchanged if $\theta<\beta-1$, and they evolve asymptotically according to a heat equation at the critical value $\theta=\beta-1$.
 In any case, the information that the process equilibrates immediately inside each box plays a fundamental role. This motivates the next definitions.
We first define the measure $\widetilde{\Pi}^n$ by
\begin{equation*}
	\widetilde{\Pi}_t^n(du)= \widetilde{\Pi}^n(\eta_t,du):=\frac{1}{nk}\sum_{i=0}^{k-1}\sum_{x=in}^{(i+1)n-1}\Big(\frac{1}{n}\sum_{y=in}^{(i+1)n-1}\eta_t(y)\Big)\delta_{x/nk}(du)\,.
\end{equation*}
That is, compared to the definition of $\Pi^n$, %for each $x$ in the $i$-th box,
instead of assigning mass $\frac{1}{nk}\eta(x)$ at point $\frac{x}{nk}\in\bb T$, we assign $\frac{1}{nk}$ times the average occupation in the $i-$th box at this point. We shall show that $\Pi^n$ and $\widetilde{\Pi}^n$ are close in a  certain sense, which indicates an immediate mixing of the process inside each box.  It remains to derive the scaling limit of $\widetilde{\Pi}^n$. Since for $\widetilde{\Pi}^n$, what we are really interested in is the evolution of $\frac{1}{n}\sum_{y=in}^{(i+1)n-1}\eta_t(y)$,  thus instead of dealing with $\widetilde{\Pi}^n$, we consider the following simplified measure.

Define $\pi_t^n$ to be the averaged  empirical measure on $\bb T$ associated to the configuration $\eta_t$:
\begin{equation}\label{pidef}
	\pi_t^n(du)\;=\; \pi^n(\eta_t,du)\;:=\;\frac{1}{k}\sum_{i=0}^{k-1}\Big(\frac{1}{n}\sum_{x=in}^{(i+1)n-1}\eta_t(x)\Big)\delta_{i/k}(du)\,.
\end{equation}
Differently from the measure $\widetilde{\Pi}^n$, here we treat each box as a whole and assign to each point $i/k$ a random mass which is $1/k$ times the average  occupation in the $i$-th box.  Since there is at most one particle allowed at each site, both $\widetilde{\Pi}^n$ and $\pi^n$ are random elements of $ \mc M$.

Let $\gamma:\bb T\to\bb [0,1]$ be a continuous density profile and let $\mu^n$ be the sequence of Bernoulli product measures of slowly varying parameter associated to the profile $\gamma$, i.e.,:
\begin{equation}\label{mun}
	\mu^n\big\{\eta: \eta(x)=1\big\}\;=\;\gamma(x/kn)\,,
\end{equation}
for all $x\in\bb T_{kn}$. Assume that the process $\{\eta_t^n:t\geq 0\}$ starts from the measure $\mu^n$.

As we discussed above, instead of deriving the scaling limit of $\Pi^n$,  we will first show that $\Pi^n$ and $\widetilde{\Pi}^n$ are close with respect to the weak topology, then we obtain the scaling limit of $\pi^n$. In fact the process equilibrates immediately inside each box  under time scale $k^2 n^{2+\delta}$ for any $\delta>0$. This is the content of the next proposition.

\begin{proposition}[Immediate equilibrium inside each box]\label{mixing}
	\qquad Assume that $k$ is fixed or\break $k=k(n)\uparrow\infty$. Fix $\theta>0$ and speed up the process by $k^2n^{2+\theta}$. Then, for any $t\in[0,T]$ and any function $G\in C^{1,2}([0,T]\times \bb T)$,
	$$\limsup_{n\to\infty}\bb E_{\mu^n}\Big[\Big|\int_0^t \frac{1}{nk}\sum_{i=0}^{k-1}\sum_{x=in}^{(i+1)n-1}G_s\big(\pfrac{x}{nk}\big)\Big( \eta_s(x)\,-\,\frac{1}{n}\sum_{y=in}^{(i+1)n-1}\eta_s(y)\big) \Big)ds\Big|\Big]\;=\;0\,.$$
\end{proposition}

The next two results are the laws of large numbers for the averaged empirical measure under the condition that $k$ is either constant or not too large.
\begin{theorem}\label{finitesmall}
	Fix $T>0$ and $k\in\bb N$. We speed up the process by $k^2n^{2+\theta}$ with $\theta<\beta-1$.  Then for every $t\in[0,T]$, the sequence of random measures $\pi^n_{t}$ converges in probability to the (time-independent) measure
	on the continuous torus $\bb T$
	\begin{equation*}
		\pi_t(du)\;=\;\sum_{i=0}^{k-1}\overline{\gamma}(i)\delta_{i/k}(du)
	\end{equation*}
	 where $\overline{\gamma}:\bb T_k\to[0,1]$ is the  function given by
	$\overline{\gamma}(i)=k\int_{i/k}^{(i+1)/k}\gamma(u)du.$
\end{theorem}
\begin{theorem}\label{finite}
	Fix $T>0$ and speed up the process by $k^2n^{1+\beta}$.
	Assume that $k$ is fixed.  Then, for every $t\in[0,T]$, the sequence of random measures $\pi^n_{t}$ converges in probability to the measure
	\begin{equation}\label{defempd}
		\pi_t(du)\;=\;\sum_{i=0}^{k-1}\rho_t(i)\delta_{i/k}(du)
	\end{equation}
	on the continuous torus $\bb T$, where $\rho_t(\cdot)$ is the unique solution of the following discrete heat equation:
	\begin{equation}\label{dHDL}
		\begin{cases}
			\partial_t\rho_t(i)=\alpha \Delta_k\rho_t(i), \quad i\in \bb T_k, \vspace{4pt}\\
			\rho_0(i)=k\int_{i/k}^{(i+1)/k}\gamma(u)du,  \quad i\in \bb T_k
		\end{cases}
	\end{equation}
	where $\Delta_k$ is the discrete Laplacian, i.e., $\Delta_k\rho_t(i):= k^2\big[ \rho_t(i+1)+\rho_t(i-1)-2\rho_t(i)\big]$.
\end{theorem}

\begin{figure}[H]
	\centering
	\begin{tikzpicture}[scale=1,smooth];
	%%%%%%%%%%%%% axes %%%%%%%%%%
	\draw[->] (0,-0.5)--(0,4) node[anchor=east]{$\gamma$};
	\draw[->] (-0.5,0)--(13,0) node[anchor=north]{$u$};

	%%%% tangente %%%%%%%%%%
	\draw[black = solid,  very thick ] plot   [domain=0:12](\x,{1+10*(\x/12)*(1-\x/12)});

	\foreach \i in {1,...,24} {
		\draw (\i/2, 0.1) -- (\i/2, -0.1) ;
	}

	\draw (0.5,-0.1) node[below]{\small $\frac{1}{kn}$};
	%%%% linha de suporte  %%%%
	\draw[densely dashed] (2,2.8) -- (2,-0.1) node[below]{\small $\frac{n}{kn}$};
	\draw[very thick] (2,1.72) -- (0,1.72) node[left]{\tiny $\displaystyle k\int_0^{\frac{1}{k}}\!\!\!\gamma(u)du$};

	\draw[dashed] (2,2.8) -- (0,2.8) node[left]{\tiny $\displaystyle k\int_{\frac{1}{k}}^{\frac{2}{k}}\!\!\!\gamma(u)du$};

	\draw[densely dashed] (4,3.4) -- (4,-0.1) node[below]{\small $\frac{2n}{kn}$};
	\draw[very thick] (2,2.8) -- (4,2.8);

	\draw[densely dashed] (6,3.4) -- (6,-0.1);
	\draw[very thick] (4,3.4) -- (6,3.4);

	\draw[very thick] (6,3.4) -- (8,3.4);
	\draw[very thick] (8,2.8) -- (10,2.8);

	\draw[very thick] (10,1.72) -- (12,1.72);

	\draw[densely dashed] (8,3.4) -- (8,-0.1);
	\draw[densely dashed] (10,2.8) -- (10,-0.1);
	\draw[densely dashed] (12,1.72) -- (12,-0.1)node[below]{\small $\frac{kn}{kn}$};

	\end{tikzpicture}
	\bigskip

	\caption{The graph of the macroscopic profile $\gamma:\bb T \to [0,1]$ and its averages in each of the $k$ boxes of size $n$. Note the embedding of the discrete torus $\bb T_{kn}$ into the continuous torus $\bb T=[0,1)$.}\label{fig3}
\end{figure}
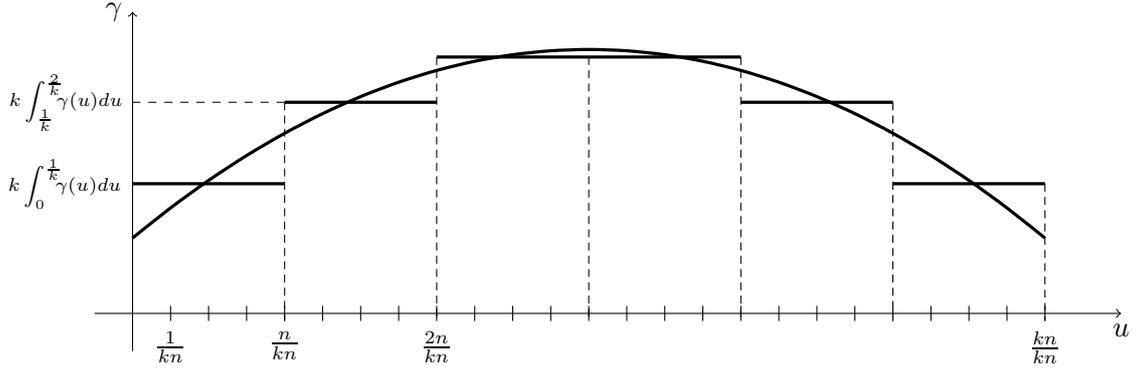

%We now give the definition of the weak solution of the heat equation appearing in the last theorem.  Given a function $G:\bb T \to \mathbb{R}$ and $\pi\in\mc M$, denote by $\langle \pi^n,G \rangle$ the integral of $G$ with respect to~$\pi$. For positive integers $m$ and $n$, denote by $C^{m,n}([0,T]\times \bb T)$ the space of continuous function with $m$ continuous derivative in time and $n$ continuous derivative in space.  $\{ \rho_t(\cdot): 0\leq t\leq T\}$ defined on $\bb T_k$ (resp. $\bb T$) is said to be the weak solution of equation \eqref{dHDL} (resp. \eqref{conHDL}) if for any function $G\in C^{1,2}([0,T]\times \bb T)$ and the empirical measure defined in \eqref{defempd} (resp. \eqref{defempcon} ),
%$$\langle \pi_t, G_t\rangle \;=\; \langle \pi_0, G_0\rangle \,-\, \int_0^t \langle \pi_s, \partial_sG_s +\alpha \Delta G_s \rangle ds\,. $$

Assume that $\gamma$ is smooth and additionally that there exists  $\varepsilon_0\in(0,1/2)$ such that $\varepsilon_0< \gamma(u)<1-\varepsilon_0$ for every $u\in\bb T$. By the maximum principle, if $\rho_t$ is the solution of continuous heat equation with initial condition $\gamma$, then $\varepsilon_0< \rho_t(u)<1-\varepsilon_0$ for every $u\in\bb T$.

\begin{theorem}\label{infinite}
	Fix $T>0$ and  speed up the process by $k^2n^{1+\beta}$. Assume that $k=k(n)\uparrow\infty$ as $n\uparrow\infty$.  Then, for every $t\in[0,T]$ and every $G\in C^2(\bb T)$,
	$$\lim_{n\to\infty}\bb E_{\mu^n}\Big[\Big| \langle \pi_{t}^n,G \rangle\,-\,\int_{\bb T}G(u)\rho_t(u)du\Big|\Big]\;=\;0\,,$$
	where $\rho_t(\cdot)$ is the strong solution of the heat equation:
	\begin{equation}\label{cHDL}
		\begin{cases}
			\partial_t\rho_t(u)=\alpha\Delta \rho_t(u), \quad u\in \bb T \\
			\rho_0(u)=\gamma(u),  \quad u\in \bb T.
		\end{cases}
	\end{equation}
\end{theorem}
The proofs of Theorem~\ref{finite} and Theorem~\ref{finitesmall} are based on Varadhan's entropy method, see~\cite{kl} for details on the subject. The same method could be used to achieve Theorem~\ref{infinite}, but only if $k=o(n^{\frac{\beta-1}{2}})$. This restriction arises from the proof of the replacement lemma (see Lemma~\ref{repinf}), which can only be established under the assumption $k=o(n^{\frac{\beta-1}{2}})$.

To drop this restriction on the growth of $k$, we employ the refined Yau's entropy method introduced by Jara and Menezes in \cite{jm}, which requires some regularity on the initial profile but, as we shall see, does not impose any bound on how fast
$k$ can grow. However, this method does not cover the case where $k$ is fixed. This explains why two different approaches are needed: one for $k$ fixed and one for $k\to\infty$.

\section{Fixed number of boxes}\label{s3}
In this section we prove the following results concerning $k$ fixed:  Proposition~\ref{mixing}, which says that the systems immediately equilibrates inside each box,  Theorem~\ref{finitesmall}, which says that there is no time evolution if the speed is strictly smaller than $k^2n^{1+\beta}$, and Theorem~\ref{finite}, which describes the limit of the system as being governed by the discrete heat equation when the speed is  $k^2n^{1+\beta}$.
\subsection{The replacement lemma}
It is straightforward to check that the process has a family of reversible invariant measures given by the Bernoulli product measures of constant parameter.  Let us denote by $\nu^n$ the Bernoulli product measure of constant parameter equal to $1/2$ on $\Omega_{nk}$. Given a function $f:\Omega_{nk}\to\bb R$, define the associated Dirichlet form with respect to a measure $\nu$ by
\begin{equation*}
\mf D(f;\nu)\,:=\, \int \sum_{x\in \bb T_{nk}}\,\xi^{nk}_{x,x+1}\,[f(\eta^{x,x+1})-f(\eta)]^2 d\nu(\eta)\,.
\end{equation*}
If $\nu$ is the invariant measure $\nu^{n}$, we write $\mf D(f; \nu^{n})$ simply  as $\mf D(f)$. We also define the Dirichlet form corresponding to the jump over a bond $(x,x+1)$ by
\begin{equation*}
\mf D_{x,x+1}(f;\nu)\,:=\, \int \,[f(\eta^{x,x+1})-f(\eta)]^2 d\nu(\eta)\,.
\end{equation*}
Again, if $\nu$ is the invariant measure $\nu^{n}$, we omit it in the notation.
With these notations, we have
\begin{equation*}
\mf D(f;\nu)\;=\; \sum_{x\in\bb T_{nk}} \xi^{nk}_{x,x+1}\mf D_{x,x+1}(f;\nu)\,.
\end{equation*}
Note that the definition above of Dirichlet forms \textbf{does not include} the acceleration of the process.
The next lemma tells us that for the speeded up process, $k$ times the total cost of replacing the end points of every box by the average occupation  is negligible. Of course, $k$ cannot be extremely large: the correct assumption is to assume
that $k$ is either  fixed or at most  a certain power of $n$, as we can see below. We note that only the case where $k$ is fixed is necessary in this section; however, we include $k=o(n^{\frac{\beta-1}{2}})$ in the statement below since it will come in handy when we study the large deviations of the problem in the future.
\begin{lemma}\label{repinf}
	Assume that $k$ is fixed or $k=k(n)=o(n^{\frac{\beta-1}{2}})$ and  the process is speeded up by $k^2n^{1+\beta}$. Then, for any $t\in[0,T]$ and any function $G\in C^{1,2}([0,T]\times\bb T)$,
	$$\limsup_{n\to\infty}\bb E_{\mu^n}\Big[\Big|\int_0^t k\sum_{i=0}^{k-1}G_s\big(\pfrac{i}{k}\big)\Big( \eta_s\big((i-1)n\big)\,+\,\eta_s(in-1)\,-\,\frac{2}{n}\sum_{x=(i-1)n}^{in-1}\eta_s(x) \Big)ds\Big|\Big]\;=\;0\,.$$
\end{lemma}
\begin{proof}
	Denote the term inside the time integral by $V_s$. By the entropy inequality, the expectation in the lemma is less than or equal to
	\begin{equation}\label{repa1}
		\frac{H(\mu^n|\nu^{n})}{Akn}\,+\, \frac{1}{Akn}\log \bb E_{\nu^{n}}\Big[\exp\Big\{Akn\Big|\int_0^t V_s\, ds\Big|\Big\}\Big]
	\end{equation}
	for any $A>0$.
	A simple computation shows that $H(\mu^n|\nu^{n})=O(nk)$.  Therefore the first term in  \eqref{repa1} vanishes as $n\to\infty$ if we let $A\to\infty$, which will be made \textit{a posteriori}. Using the inequality $e^{|x|}\leq e^x+e^{-x}$ and the fact that
	\begin{equation*}
		\limsup_n\frac{1}{n}\log(a_n+b_n)\;=\;\max\Big\{\limsup_n\frac{1}{n}\log a_n,\,\limsup_n\frac{1}{n}\log b_n \Big\}\,,
	\end{equation*}
	for all sequences $\{a_n\}$ and $\{b_n\}$ of positive numbers, we see that we just need to estimate the second term in \eqref{repa1} without the absolute value inside the exponential. Using the Feymann-Kac formula (see \cite[Lemma 7.2. p.336]{kl}), we reduce the problem to showing that
	\begin{equation}\label{repa2}
		\limsup_{n\to\infty}\; t\cdot \sup_{0\leq s\leq t}\sup_{f \,\text{density}}\Big\{ \int V_s f d\nu^{n}\,-\, \frac{kn^{\beta}}{A} \big\<\sqrt{f},(-\mc L_n) \sqrt{f}\big\>_{\nu^{n}}\Big\}\;=\;0\,.
	\end{equation}
	 To obtain the above expression we made use of the fact that the process is speeded up by $k^2n^{1+\beta}$, so that the generator of the process is $k^2n^{1+\beta}\mc L_n$. We moreover note that the supremum in time above, appears as a consequence of the fact that the test function $G_s$ depends on time.

	Let us first deal with the replacement on the box $\{0,1,\ldots,n-1\}$. Note that
	$$\eta(0)\,-\,\frac{1}{n}\sum_{x=0}^{n-1}\eta(x)\;=\;\sum_{x=1}^{n-1}\frac{n-x}{n}\big(\eta(x-1)-\eta(x)\big),$$
	and
	$$\eta(n-1)\,-\,\frac{1}{n}\sum_{x=0}^{n-1}\eta(x)\;=\;\sum_{x=1}^{n-1}\frac{x}{n}\big(\eta(x)-\eta(x-1)\big).$$
We thus get
	$$\eta(0)\,+\,\eta(n-1)\,-\,\frac{2}{n}\sum_{x=0}^{n-1}\eta(x)\;=\; \sum_{x=1}^{n-1}\frac{n-2x}{n}\big(\eta(x-1)-\eta(x)\big)\,.$$
	For each $1\leq x\leq n-1$, performing the change of variables $\eta\mapsto\eta^{x,x-1}$, for which $\nu^n$ is invariant, we have that
	$$\int \big(\eta(x-1)\,-\,\eta(x)\big)f d\nu^{n}\;=\;\int \eta(x)\big(f(\eta^{x-1,x})-f(\eta)\big)d\nu^{n}.$$
	Writing
	$f(\eta^{x-1,x})-f(\eta)=\big\{\sqrt{f(\eta^{x-1,x})}\,-\,\sqrt{f(\eta)}\big\}\big\{\sqrt{f(\eta^{x-1,x})}\,+\,\sqrt{f(\eta)}\big\}$ and using Young's inequality, we obtain that for any $B>0$,
	$$\int G_s(0) \frac{n-2x}{n}\big(\eta(x-1)-\eta(x)\big)fd\nu^{n}\,\leq \, \frac{1}{B} \mf D_{x-1,x}(\sqrt{f})\,+\,B \|G\|_\infty^2\,.$$
	Summing $x$ from $1$ to $n-1$, we have that
	\begin{equation}\label{sumbound}
		\int G_s(0)\Big\{\eta(0)\,+\,\eta(n-1)\,-\,\frac{2}{n}\sum_{x=0}^{n-1}\eta(x)\Big\}f d\nu^{n}\;\leq\; \sum_{x=1}^{n-1}\frac{1}{B}\mf    D_{x-1,x}(\sqrt{f})\, +\, Bn\|G\|_\infty^2.
	\end{equation}
	We can do the same estimate for the replacement on the box $\{in,\ldots, (i+1)n-1\}$ for every $1\leq i\leq k-1$. Summing $i$ from $0$ to $k-1$, we obtain that for any $B>0$,
	$$\int V_s fd\nu^{n}\;\leq\; \frac{k}{B}\mf D(\sqrt{f})\,+\, Bk^2n\|G\|_\infty^2.$$
	To finish the proof, it remains to choose $A=\frac{\sqrt{n^{\beta-1}}}{k}$ and $B=(k^2n^{\beta+1})^{-1/2}$.

\end{proof}
Following almost the same argument, one can easily obtain the next result.
\begin{corollary}\label{corsmall}
	Assume that $k$ is fixed and fix $\theta >0$. We speed up the process by $k^2n^{2+\theta}$. Then for any $t\in[0,T]$, any function $G\in C^{1,2}([0,T]\times\bb T)$,
	$$\limsup_{n\to\infty}\bb E_{\mu^n}\Big[\Big|\int_0^t k\sum_{i=0}^{k-1}G_s\big(\pfrac{i}{k}\big)\Big( \eta_s\big((i-1)n\big)\,+\,\eta_s(in-1)\,-\,\frac{2}{n}\sum_{x=n(i-1)}^{in-1}\eta_s(x) \Big)ds\Big|\Big]\;=\;0\,.$$
\end{corollary}

Here we also give the proof of Proposition \ref{mixing}, because it is quite similar to the proof of the previous lemma.

\begin{proof}[Proof of Proposition \ref{mixing}]
	Using almost the same arguments as in the proof of Lemma \ref{repinf}, we see that the expectation in the statement is bounded from above by
	$$ C(T)\Big\{ \frac{1}{A} \,+\,\frac{1}{kB}\mf D(\sqrt{f})\,+\, \frac{1}{4}B n\|G\|_\infty^2\,-\,\frac{k n^{1+\theta}}{A}\mf  D(\sqrt{f}) \Big\} $$
	for any $A, B>0$. Choosing $A=kn^{\theta/2}$ and $B=k^{-1}n^{-(2+\theta)/2}$, we finish the proof.
\end{proof}

\subsection{Tightness}
Let $k$ be fixed. Denote by $\bb Q_{\mu^n}$ the probability measure on $D([0,T],\mc M)$ corresponding to  the stochastic process $\pi^n$ defined in \eqref{pidef}, speeded up by $k^2 n^{1+\beta}$ and starting from $\mu^n$.

\begin{proposition}
	The sequence $\{\bb Q_{\mu^n}\}_{n\in\bb N}$  tight with respect to the uniform topology.
\end{proposition}
\begin{proof}
	To prove the tightness of $\pi^n$, it is enough to prove the tightness of $\langle \pi^n, G\rangle $ for any smooth function $G:\bb T \to \bb R$.  Consider the Dynkin's martingale
	\begin{equation}\label{dynmar}
		M_t^n(G)\;=\;  \langle \pi^n_t, G\rangle-\langle \pi^n_0, G\rangle - k^2 n^{1+\beta}\int_0^t \mc L_n \langle \pi^n_s, G\rangle ds\,.
	\end{equation}
	A straightforward computation gives that
	\begin{equation}\label{LnpiG}
		k^2n^{1+\beta}\,\mc L_n\<\pi^n_s, G\> \;=\;\frac{k^2n^{1+\beta}}{kn}\sum_{i=0}^{k-1}G\big(\frac{i}{k}\big)\frac{\alpha}{n^\beta}\Big[\eta((i+1)n)-\eta((i+1)n-1)+\eta\big(in-1\big)-\eta\big(in\big)\Big].
	\end{equation}
	Performing a summation by parts and recalling that $G$ is smooth, one can check that
	\begin{equation}\label{integral_part}
	\big\vert k^2n^{1+\beta}\,\mc L_n\<\pi^n_s, G\> \big\rvert \;\leq\; C
	\end{equation}
	for some constant $C$ depending on $\alpha$ and $G(\cdot)$ but not depending on $n$.
	The quadratic variation of the above martingale can be explicitly computed:
	\begin{equation*}
		\begin{split}
			\langle M^n \rangle_t =\, &\frac{k^2n^{1+\beta}}{k^2n^2}  \int_0^t \frac{\alpha}{n^\beta}\sum_{i=0}^{k-1}  \eta_s((i+1)n-1)\big(1-\eta_s((i+1)n)\big)\big[ G\big(\frac{i+1}{k}\big) -G\big(\frac{i}{k}\big)\big]^2 ds\\
			+\,& \frac{k^2n^{1+\beta}}{k^2n^2} \int_0^t \frac{\alpha}{n^\beta}\sum_{i=0}^{k-1}  \eta_s(in)\big(1-\eta_s(in-1)\big) \big[G\big(\frac{i-1}{k}\big) - G\big(\frac{i}{k}\big)\big]^2  ds.
		\end{split}
	\end{equation*}
	This implies that $\langle M^n \rangle_t \leq Ctn^{-1}k^{-1}$ where $C$ is a constant depending on $\alpha$ and $G(\cdot)$.

	In view of \eqref{dynmar}, the tightness of $\langle \pi^n, G\rangle $ follows from the tightness of the martingale $M_t^n(G)$ and the tightness of the additive functional $k^2 n^{1+\beta}\int_0^t \mc L_n \langle \pi^n_s, G\rangle ds$. The martingale is tight with respect to the uniform topology because, by Doob's inequality and the bound $\langle M^n \rangle_t \leq Ctn^{-1}k^{-1}$,  for every $\tau>0$,
	\begin{equation*}
		\lim_{n\to\infty}\bb P_{\mu_n}\Big[ \sup_{0\leq t\leq T}\big\lvert M_t^n(G)  \big\rvert >\tau\Big]\;\leq\;	\lim_{n\to\infty}\frac{1}{\tau}\bb E_{\mu_n}\Big[ \big\lvert M_T^n(G)  \big\rvert^2\Big]\;=\;0\,.
	\end{equation*}
	On the other hand,  the additive functional part is relatively compact in the uniform topology in view of \eqref{integral_part} and the Arzel\'a-Ascoli Theorem, thus tight by Prohorov's Theorem.
	\end{proof}
Let us use the same notation $\bb Q_{\mu^n}$ to denote the probability measure on $D([0,T], \mathcal{M})$ corresponding to the stochastic process $\pi^n$, as defined in \eqref{pidef}, but accelerated by $k^2 n^{2+\delta}$ with $\delta < \beta - 1$, and starting from $\mu^n$. By applying similar arguments as before, we can also conclude that the sequence $\{\bb Q_{\mu^n}\}_{n \in \bb N}$ is tight with respect to the uniform topology.

\subsection{Characterization of limit points}
Let $\bb Q^*$ be the weak limit
of some convergent subsequence  $\{\bb Q_{\mu^{n_j}}\}_{j\in \bb N}$ of the sequence  $\{\bb Q_{\mu^n}\}_{n\in \bb N}$. To simplify the notation, we will denote this subsequence by $\{\bb Q_{\mu^n}\}_{n \in \bb N}$.

First note that, as a probability distribution on the discrete set $\bb T_k$, the measure $\bb Q^*$ is actually absolutely continuous with respect to the uniform distribution.

  We claim that every limit point $\bb Q^*$ is concentrated on trajectories that at time $0$ are equal to
$$\sum_{i\in\bb T_k}\Big( k\int_{i/k}^{(i+1)/k}\gamma(u)du\Big) \delta_{i/k}\,.$$
Indeed, by the weak convergence of $\bb Q_{\mu^n}$ towards $\bb Q^*$, and the fact that $\mu^n$ is a product Bernoulli measure associated to the profile $\gamma$ as defined in~\eqref{mun} one can show that, for every $\varepsilon>0$,
\begin{equation*}
	\begin{split}
		&\bb Q^*\Big[\Big\lvert \frac{1}{kn}\sum_{i=0}^{k-1}\Big(\sum_{x=in+1}^{(i+1)n}\eta_0(x)\Big) G\big(\frac{i}{k}\big) \,-\, \frac{1}{k}\sum_{i\in\bb T_k}\Big( k\int_{i/k}^{(i+1)/k}\gamma(u)du\Big) G\big(\frac{i}{k}\big)  \Big\rvert >\varepsilon\Big]\\
		\leq\,& \liminf_{n\to\infty} \bb Q_{\mu^n}\Big[\Big\lvert \frac{1}{kn}\sum_{i=0}^{k-1}\Big(\sum_{x=in+1}^{(i+1)n}\eta_0(x)\Big) G\big(\frac{i}{k}\big) \,-\, \frac{1}{k}\sum_{i\in\bb T_k}\Big( k\int_{i/k}^{(i+1)/k}\gamma(u)du\Big) G\big(\frac{i}{k}\big)  \Big\rvert >\varepsilon\Big]\\
		=\,& \liminf_{n\to\infty} \mu^n \Big [\Big\lvert \frac{1}{kn}\sum_{i=0}^{k-1}\Big(\sum_{x=in+1}^{(i+1)n}\eta_0(x)\Big) G\big(\frac{i}{k}\big) \,-\, \frac{1}{k}\sum_{i\in\bb T_k}\Big( k\int_{i/k}^{(i+1)/k}\gamma(u)du\Big) G\big(\frac{i}{k}\big)  \Big\rvert >\varepsilon\Big]\;=\;0\,.\\
	\end{split}
\end{equation*}

Let us first consider the case that $\bb Q_{\mu^n}$ is  the probability measure on $D([0,T], \mc M)$ corresponding to  the stochastic process $\pi^n$ defined in \eqref{pidef}, speeded up by $k^2 n^{1+\beta}$.  For any function $G\in C^{2}( \bb T)$,  we  claim  that
\begin{equation}\label{Q1}
	\bb Q^{*}\Big\{\,\pi:\,
	\<\pi_t, G\> \,-\,
	\<\pi_0,G\> \,-\,
	\int_0^t \langle \pi_s,\alpha \Delta_k G \rangle ds \;=\;0,\,\forall \, t\in[0,T]\, \Big\}\;=\;1\,.
\end{equation}
To prove this, it suffices to show that
\begin{equation*}
	\bb Q^{*} \Big\{\,\pi: \, \sup_{0\le t\le T}
	\Big\vert \<\pi_t, G \> \,-\,
	\<\pi_0,G \> \,-\,
	\int_0^t \langle \pi_s, \alpha \Delta_k G \rangle ds\, \Big\vert
	\, > \, \delta\, \Big\}\,=0\,,
\end{equation*}
for every $\delta >0$. Since the supremum is a continuous function in the Skorohod metric, by Portmanteau's Theorem, the probability above is smaller or equal  than
\begin{equation*}
	\liminf_{n\to\infty} \bb Q_{\mu_n} \Big\{\,\pi: \, \sup_{0\le t\le T}
	\Big\vert \<\pi_t, G\> \,-\,
	\<\pi_0,G\> \,-\,
	\int_0^t \langle \pi_s, \alpha \Delta_k G \rangle ds \,\Big\vert
	\, > \, \delta\,\Big\}\,.
\end{equation*}
Since $\bb Q_{\mu^n}$  is the measure on the space $\mc D([0,T], \mc M)$ induced by $\bb P_{\mu_n}$ via the empirical measure, we can rewrite the expression above as
\begin{equation*}
	\liminf_{n\to\infty} \bb P_{\mu^n} \Big\{\,\eta_\bola\,:\, \sup_{0\le t\le T}
	\Big\vert \<\pi_t^n, G \> \,-\,
	\<\pi_0^n,G \> \,-\,
	\int_0^t  \,  \langle \pi^n_s,\alpha \Delta_k G \rangle \,ds \,\Big\vert
	\, > \, \delta\,\Big\}\,.
\end{equation*}

Adding and subtracting $k^2n^{1+\beta}\,\mc L_n\<\pi^n_s, G\>$ to the integral term above, we can see that the previous expression is bounded from above by the sum of
\begin{equation}\label{eq413}
	\begin{split}
		&\limsup_{n\to\infty} \bb P_{\mu^n} \Big\{ \, \sup_{0\le t\le T}
		\Big\vert \<\pi_t^n, G \> \,-\,
		\<\pi_0^n,G \> \,-\,
		\int_0^t  \,  k^2n^{1+\beta}\,\mc L_n\<\pi^n_s, G\> \,ds\, \Big\vert
		\, > \, \delta/2\, \Big\}\\
	\end{split}
\end{equation}
and
\begin{equation}\label{eq414}
	\begin{split}
		& \limsup_{n\to\infty} \bb P_{\mu^n} \Big\{\, \sup_{0\le t\le T} \Big\vert \int_0^t
		\big(k^2n^{1+\beta}\,\mc L_n\<\pi^n_s, G\> - \alpha \<\pi_s^n ,\Delta_k G \>\big)\,ds \,\Big\vert
		\, > \, \delta/2 \,\Big\}\,.
	\end{split}
\end{equation}
Consider the Dynkin's martingale
$$M_t^n\;=\; \<\pi_t^n, G \> -
\<\pi_0^n,G \> -
\int_0^t  \Big\{  k^2n^{1+\beta}\,\mc L_n\<\pi^n_s, G\>\Big\}ds\,.$$
As in the proof of tightness, one can bound the quadratic variation of $M_t^n$ by $Ctn^{-1}k^{-1}$ and then using it to show that  expression \eqref{eq413} is null.

We now show that \eqref{eq414} also vanishes. Recalling \eqref{LnpiG}, from the Replacement Lemma~\ref{repinf}, we have that
\begin{equation*}
	\begin{split}
		\limsup_{n\to\infty}\, &\bb P_{\mu^n} \Big\{\, \sup_{0\le t\le T} \Big\vert \int_0^t
	k^2n^{1+\beta}\,\mc L_n\<\pi^n_s, G\>\\ -
	&\alpha k\sum_{i=0}^{k-1}G\big(\frac{i}{k}\big)\Big[\frac{1}{n}\sum_{x=(i+1)n}^{(i+2)n-1}\eta_s(x)+\frac{1}{n}\sum_{x=(i-1)n}^{in-1}\eta_s(x)-\frac{2}{n}\sum_{x=in}^{(i+1)n-1}\eta_s(x)\Big]\,ds \,\Big\vert
		\, > \, \delta/4 \,\Big\}=0\,.
	\end{split}
\end{equation*}
Performing a summation by parts, it is then enough to prove that
\begin{equation*}
\begin{split}
\limsup_{n\to\infty} \bb P_{\mu^n} \Big\{\, \sup_{0\le t\le T} \Big\vert \int_0^t
&\alpha k\sum_{i=0}^{k-1}\eta_s(x)\big\{G\big(\frac{i+1}{k}\big)+G\big(\frac{i-1}{k}\big)-2G\big(\frac{i}{k}\big)\big\}\\
- &\alpha \<\pi_s^n ,\Delta_k G \>\big)\,ds \,\Big\vert
\, > \, \delta/4 \,\Big\}\;=\;0\,.
\end{split}
\end{equation*}
This holds trivially because the expression inside the absolute value is zero.
Since $k$ is fixed,  there is a unique solution of $\eqref{dHDL}$, and hence uniqueness follows.

We now consider the case that $k$ is fixed and $\bb Q_{\mu^n}$ is the probability measure on $D([0,T],\mc M)$ corresponding to  the stochastic process $\pi^n$ defined in \eqref{pidef}, speeded up by $k^2 n^{2+\theta}$ with $\theta<\beta-1$. In this case we claim that
\begin{equation}
	\bb Q^{*}\Big\{\,\pi:\,
	\<\pi_t, G\> \,-\,
	\<\pi_0,G \> \;=\;0,\,\forall \, t\in[0,T]\, \Big\}\;=\;1\,
\end{equation}
for any function $G\in C^{2}(\bb T)$.
Similar to what we did before, it is sufficient to show that
\begin{equation}\label{eq415}
	\begin{split}
		&\limsup_{n\to\infty} \bb P_{\mu^n} \Big\{ \, \sup_{0\le t\le T}
		\Big\vert \<\pi_t^n, G \> \,-\,
		\<\pi_0^n,G \> \,-\,
		\int_0^t  \,  k^2n^{2+\theta}\,\mc L_n\<\pi^n_s, G\> \,ds\, \Big\vert
		\, > \, \delta/2\, \Big\}=0\,,\\
	\end{split}
\end{equation}
and
\begin{equation}\label{eq416}
	\limsup_{n\to\infty} \bb P_{\mu^n} \Big\{\, \sup_{0\le t\le T} \Big\vert \int_0^t
	k^2n^{2+\theta}\,\mc L_n\<\pi^n_s, G\> \,ds \,\Big\vert
	\, > \, \delta/2 \,\Big\}=0\,.
\end{equation}
To prove \eqref{eq415}, we can define the martingale
$$M_t^n\;=\; \<\pi_t^n, G \> \,-\,
\<\pi_0^n,G \> \,-\,
\int_0^t  \,k^2n^{2+\theta}\,\mc L_n\<\pi^n_s, G\>\,ds\,$$
and  examine its quadratic variation, which vanishes as $n\to\infty$. In order to prove \eqref{eq416}, using  Corollary~\ref{corsmall} and performing a summation by parts, it is enough to guarantee that
\begin{align*}
	\limsup_{n\to\infty} \bb P_{\mu^n} \Big\{\, \sup_{0\le t\le T} \Big\vert \int_0^t
	&\frac{\alpha k}{n^{\beta-\theta-1}}\sum_{i=0}^{k-1}\frac{1}{n}\Big(\sum_{x=in}^{(i+1)n-1}\eta_s(x)\Big)\big\{G\big(\frac{i+1}{k}\big)+G\big(\frac{i-1}{k}\big)-2G\big(\frac{i}{k}\big)\big\}ds\Big\vert\\
	& > \, \delta/2 \,\Big\}\;=\;0\,,
\end{align*}
which holds obviously because $\theta+1<\beta$. Thus, we can conclude.

%To conclude, it remains to show the uniqueness of the weak solution. For the case $k=k(n)\uparrow\infty$, the uniqueness of the weak solution to the heat equation \eqref{conHDL} is well known(see, for instance, Chapter 4 of \cite{kl}).

\section{Relative entropy method and Proof of Theorem~\ref{infinite}}\label{s4}

The purpose of this section is to prove Theorem~\ref{infinite}. The proof is based on the refined Yau's relative entropy method by Jara and Menezes. In Section~\ref{s41} we define the time-dependent reference measure and estimate the corresponding relative entropy. In Section~\ref{s42} we present some auxiliary results that will be used in Section \ref{s43}, where we estimate the terms coming from the entropy production. In Section~ \ref{s44} we use the bound on the relative entropy to derive the hydrodynamic limit.

\subsection{Main estimate}\label{s41}
Fix $T>0$. Recall that the process has been speeded up by $k^2n^{1+\beta}$. Denote, for every $t\in(0,T]$, by $\mu^n_t$ the distribution of the accelerated process at time $t$. Let us define the time-dependent reference measure $\nu_t^n$, $0\leq t\leq T$. For every $x\in\bb T_{kn}$, let $m_n(x)$ be the index of the box which $x$ belongs to.
Let $\nu^n_t$ be the Bernoulli product measure such that
\begin{equation}\label{nu}
	\nu_t^n(\eta: \eta(x)=1 )\;=\; \rho_t\big(m_n(x)\big)\;=\; \rho_t\big(\frac{i}{k}\big),\quad \forall \,\, in\leq x\leq (i+1)n-1,
\end{equation}
where $\rho_t(\cdot)$ is the solution of heat equation given in \eqref{cHDL}. Recall that we are denoting  by $\nu^n$  the uniform measure on $\Omega_{kn}$ of constant parameter equal to $1/2$.
Let $f_t^n$ be the Radon-Nikodym derivative of $\mu_t^n$ with respect to $\nu_t^n$:
$$f_t^n\;:=\;\frac{d\mu_t^n}{d\nu_t^n}$$
and let $\Psi_t^n$ be the Radon-Nikodym derivative of $\nu^n_t$ with respect to $\nu^n$, that is,
$$\Psi_t^n\;:=\;\frac{d\nu^n_t}{d\nu^n}\,.$$

The next lemma is a version by Jara and Menezes (see Lemma A.1 of \cite{jm}) of the so-called Yau's inequality (see Lemma~1.4 of Chapter~6 in \cite{kl}), which provides an upper bound for the time derivative of the relative entropy $ H'(\mu^n_t|\nu_t^n)$ in terms of $f_t^n$ and $\Psi_t^n$. We note that while the standard Yao's inequality takes the invariant measure $\nu^n$ as reference, here the reference is the Bernoulli product measure $\nu^n_t$.
\begin{lemma}[Refined Yao's Inequality, Lemma A.1 of \cite{jm}]\label{deriv}
	For every $t>0$,
	$$ \partial_t H(\mu^n_t|\nu_t^n)\;\leq\; -k^2n^{1+\beta}\mf D(\sqrt{f_t^n}; \nu_t^n)+ \int \Big\{k^2n^{1+\beta} \mc L^*_{n,t}1-\partial_t \log \Psi_t^n\Big\} f_t^n d\nu_t^n\,,$$
	where $\mc L^*_{n,t}$ is the adjoint generator of $\mc L_n$ with respect to the measure $\nu^n_t$.
\end{lemma}

Throughout this section, the constant $C$ may change from line to line, but never depends on $k$ and $n$. Since $\gamma$ is smooth, we can assume that there exists a constant $\kappa>0$ such that
$$\lvert\rho_t(u+\epsilon)-\rho_t(u)\rvert\;\leq\; \epsilon\kappa$$
for every $t\in[0,T]$, every $\epsilon\in(0,1)$ and every $u\in\bb T$.
Now we apply the previous lemma to prove the following proposition.
\begin{proposition}\label{deriv estim}
	There is a constant $C>0$ independent of $k$ and $n$ such that
	$$\partial_t  H(\mu^n_t|\nu_t^n)\;\leq\;CH(\mu^n_t|\nu_t^n)+ o(kn)\,, $$
\end{proposition}
\begin{proof}
Using Equation (A.1) in~\cite{jm} and a straightforward computation show that
	\begin{align*}
		&\mc L^*_{n,t}\mathds{1}(\eta)\\
		&\;=\; \sum_{x\in\bb T_{kn}}\xi^{nk}_{x,x+1}\Big\{\eta(x+1)\big(1-\eta(x)\big)\frac{\nu_t^n(\eta^{x,x+1})}{\nu_t^n(\eta)}\,-\, \eta(x)\big(1-\eta(x+1)\big) \Big\}\\
		&\qquad\qquad\qquad+ \sum_{x\in\bb T_{kn}}\xi^{nk}_{x,x+1}\Big\{\eta(x)\big(1-\eta(x+1)\big)\frac{\nu_t^n(\eta^{x,x+1})}{\nu_t^n(\eta)}\,-\, \eta(x+1)\big(1-\eta(x)\big) \Big\}\\
	&\;=\;\sum_{x\in\bb T_{kn}}\xi^{nk}_{x,x+1}\Big\{\eta(x+1)\big(1-\eta(x)\big)\frac{\rho_t(m_n(x))[1-\rho_t(m_n(x+1))]}{\rho_t(m_n(x+1))[1-\rho_t(m_n(x))]}\,-\, \eta(x)\big(1-\eta(x+1)\big) \Big\}\\
		&\quad+ \sum_{x\in\bb T_{kn}}\xi^{nk}_{x,x+1}\Big\{\eta(x)\big(1-\eta(x+1)\big)\frac{\rho_t(m_n(x+1))[1-\rho_t(m_n(x))]}{\rho_t(m_n(x))[1-\rho_t(m_n(x+1))]}\,-\, \eta(x+1)\big(1-\eta(x)\big) \Big\}.
	\end{align*}
	This identity can be rewritten as
	\begin{align*}
		\mc L^*_{n,t}\mathds{1} (\eta)\;=\; &\sum_{x\in\bb T_{kn}}\xi^{nk}_{x,x+1}\big\{\rho_t(m_n(x))\,-\,\rho_t(m_n(x+1))\big\}\\
		\times&\,\bigg\{\frac{\eta(x+1)[1-\eta(x)]}{\rho_t(m_n(x+1))[1-\rho_t(m_n(x))]} \,-\, \frac{\eta(x)[1-\eta(x+1)]}{\rho_t(m_n(x))[1-\rho_t(m_n(x+1))]}  \bigg\}\,.
	\end{align*}
	Using equation (A.3) in \cite{jm}, we have that
	\begin{align*}
		\mc L^*_{n,t}\mathds{1} (\eta)\;=\; &\sum_{x\in\bb T_{kn}}\xi^{nk}_{x,x+1}\big\{\rho_t(m_n(x))\,-\,\rho_t(m_n(x+1))\big\}\times\\
		&\,\big\{w_t(x+1)-w_t(x)+(\rho_t(m_n(x+1))-\rho_t(m_n(x)))w_t(x)w_t(x+1)\big\}\,,
	\end{align*}
	where
	\begin{equation}\label{eq:w}
	w_t(x)\;:=\;\frac{\eta(x)-\rho_t(m_n(x))}{\rho_t(m_n(x))(1-\rho_t(m_n(x)))}\,.
	\end{equation}

	After a summation by parts,
	\begin{equation}\label{Lnt}
		\begin{split}
			&\mc L^*_{n,t}\mathds{1} (\eta)\;=\; \\
			&\sum_{x\in\bb T_{kn}}w_t(x)\Big[ \xi^{nk}_{x-1,x}\Big\{\rho_t(m_n(x-1))-\rho_t(m_n(x))\big\}-\xi^{nk}_{x,x+1}\big\{\rho_t(m_n(x))-\rho_t(m_n(x+1))\big\}\Big]\\
			&-\sum_{x\in\bb T_{kn}}\xi^{kn}_{x,x+1}\big[\rho_t(m_n(x+1))-\rho_t(m_n(x))\big]^2w_t(x)w_t(x+1)\,.
		\end{split}
	\end{equation}
	By definition  $m_n(x)$ is a constant for every $x$ in the same box, thus the first parcel on the right hand side of \eqref{Lnt} is equal to
	\begin{equation}\label{first_1}
		\begin{split}
			&\sum_{i=0}^{k-1}w_t(ni)\frac{\alpha}{n^\beta}\Big[\rho_t\big(\frac{i-1}{k}\big)\,-\,\rho_t\big(\frac{i}{k}\big)\Big]
			+\sum_{i=0}^{k-1}w_t(n(i+1)-1)\frac{\alpha}{n^\beta}\Big[\rho_t\big(\frac{i+1}{k}\big)\,-\,\rho_t\big(\frac{i}{k}\big)\Big]\,.
		\end{split}
	\end{equation}
	The second parcel on the right hand side of \eqref{Lnt} is equal to
	\begin{equation}\label{second_1}
		-\sum_{i=0}^{k-1}\frac{\alpha}{n^\beta}\Big[\rho_t\big(\frac{i-1}{k}\big)-\rho_t\big(\frac{i}{k}\big)\Big]^2w_t(ni-1)w_t(ni)\;.
	\end{equation}
Let us now deal with $\Psi_t^n(\eta)$, which is the Radon-Nykodim derivative between $\nu^n_t$ and $\nu^n$. By means of some simple calculations, we can check that
\begin{align*}
\Psi^n_t & \;=\; \prod_{x\in \bb T_{kn}}\Big\{
	\eta(x)2\rho_t(m_n(x))+(1-\eta(x))2(1-\rho_t(m_n(x)))\Big\}\,,
	\end{align*}
	so
\begin{align*}
\log \Psi^n_t & = \sum_{x\in \bb T_{kn}} \eta(x)\log\Big(2\rho_t(m_n(x))\Big)+(1-\eta(x))\log\Big(2(1-\rho_t(m_n(x)))\Big)\,.
\end{align*}
Thus,
	\begin{align}
			\partial_t\log\Psi_t^n(\eta)\;=\;&\sum_{x\in\bb T_{kn}}w_t(x)\partial_t\rho_t(m_n(x))\notag\\
			=\;&\alpha\sum_{x\in\bb T_{kn}}w_t(x)\Delta\rho_t(m_n(x)) \notag\\
			=\;&\alpha k^2\sum_{i=0}^{k-1}\sum_{x=in}^{(i+1)n-1}w_t(x)\Big\{\rho_t\big(\frac{i+1}{k}\big)\,+\,\rho_t\big(\frac{i-1}{k}\big)\,-\,2\rho_t\big(\frac{i}{k}\big)\Big\}\,+\, O\big(\frac{n}{k}\big)\,.\label{derivative_Psi}
	\end{align}
where we have used above that $\rho$ is the solution of the heat equation given in \eqref{infinite} and the approximation of the continuous Laplacian by the discrete one. Putting together \eqref{first_1}, \eqref{second_1} and \eqref{derivative_Psi} and writing the discrete Laplacian as the difference of discrete derivatives, we obtain that
	\begin{align}
			\int \Big\{&k^2n^{1+\beta} \mc L^*_{n,t}\mathds{1}-\partial_t \log \Psi_t^n\Big\} f_t^n d\nu_t^n\notag\\
			&=\;\alpha k^2\int \sum_{i=0}^{k-1}\sum_{x=in}^{(i+1)n-1}\big\{ w_t(ni)\,-\,w_t(x)\big\}\Big[\rho_t(\frac{i-1}{k})\,-\,\rho_t(\frac{i}{k})\Big]  f_t^n d\nu_t^n\label{longsum_1}\\
			&\quad +\alpha k^2\int \sum_{i=0}^{k-1}\sum_{x=in}^{(i+1)n-1}\{w_t(n(i+1)-1)\,-\,w_t(x)\}\Big[\rho_t(\frac{i+1}{k})\,-\,\rho_t(\frac{i}{k})\Big]f_t^n d\nu_t^n\label{longsum_2}\\
			&\quad-\alpha n\int \sum_{i=0}^{k-1}k^2\Big[\rho_t(\frac{i-1}{k})\,-\,\rho_t(\frac{i}{k})\Big]^2w_t(ni-1)w_t(ni)f_t^n d\nu_t^n\label{longsum_3}\\
			&\quad +  O\big(\frac{n}{k}\big)\,.\notag
	\end{align}
	In order to finish the proof we need to estimate \eqref{longsum_1}--\eqref{longsum_3} in terms of the Dirichlet form so that we can apply Lemma \ref{deriv}. This is the content of the technical Lemmas \ref{p1} and \ref{p2}, which will be formulated and proven in Section \ref{s43}. This finishes the proof.
\end{proof}

\begin{corollary}\label{entropy bound}
	For every $0\leq t\leq T$, it holds that
	$H(\mu_t^n|\nu_t^n)=o(kn).$
\end{corollary}
\begin{proof}
	As long as we can show that $$H(\mu^n|\nu_0^n)\;=\;o(kn)\,,$$ this corollary follows immediately from Proposition \ref{deriv estim} and Gronwall's inequality.
	In fact we are going to prove a sharper bound: that there exists a constant $C>0$ such that
	\begin{equation}\label{time0}
		H(\mu^n|\nu_0^n)\;\leq \; Cn\,.
	\end{equation}
	Note that this bound is indeed sharper, since we are in the regime where $k$ tends to infinity with $n$.
	By the relative entropy's formula,
	\begin{equation*}
		\begin{split}
			H(\mu^n|\nu_0^n)\;=\;&\sum_{\eta\in\Omega_{nk}} \mu^n(\eta) \log \frac{\mu^n(\eta)}{\nu_0^n(\eta)}\;\leq\; \sup_{\eta\in\Omega_{nk}} \log \frac{\mu^n(\eta)}{\nu_0^n(\eta)} \\
			=\, &\sup_{\eta\in\Omega_{nk}} \sum_{i\in\bb T_k}\sum_{x=in}^{(i+1)n-1}\Big\{\eta(x) \log\frac{\gamma(x/nk)}{\gamma(i/k)}\,+\, (1-\eta(x)) \log\frac{1-\gamma(x/nk)}{1-\gamma(i/k)}\Big\}\,.\\
		\end{split}
	\end{equation*}
	Since $\gamma$ is smooth and $\varepsilon_0 <\gamma(\cdot)<1-\varepsilon_0$, there exists a constant $C=C(\varepsilon_0,\kappa)$, independent of $n$ and $k$, such that
	$$ \log\frac{\gamma(x/nk)}{\gamma(i/k)}\;\leq\; \log\Big(1+ \frac{\kappa/k}{\gamma(i/k)} \Big)\;\leq\; C(\varepsilon_0,\kappa) k^{-1}\,.$$
	Similarly, the following inequality also holds:
	$$\log\frac{1-\gamma(x/nk)}{1-\gamma(i/k)} \;\leq\; C(\varepsilon_0,\kappa) k^{-1}\,.$$
	These two estimates imply the desired estimate \eqref{time0}.
\end{proof}

\subsection{Some auxiliary results}\label{s42}
In this subsection we present some results that will be used in the next subsection.  First we state an integral by parts formula, which is a simple consequence of Lemma E.3 in \cite{jm}.
\begin{lemma}\label{single}
	Fix a measure $v$ on $\Omega_{nk}$. Fix $x\in\bb T_{kn}$ and  let $h:\Omega_{kn}\to\bb R$ be such that $h(\eta)=h(\eta^{x,x+1})$ for every $\eta\in\Omega_{nk}$. Then, for any $A>0$ and any density $f$ with respect to $\nu$,
	\begin{equation*}
		\begin{split}
			\int h(w(x+1)-w(x)) f d\nu\;\leq\; &A\, \mf D_{x,x+1}(\sqrt{f},\nu)\,+\,\frac{4}{A\varepsilon_0}\int h^2 fd\nu\\
			\,& - \Big[\rho(m_n(x+1))-\rho(m_n(x))\Big]\int hw(x)w(x+1)fd\nu\,,
		\end{split}
	\end{equation*}
	where we recall that
	$\mf D_{x,x+1}(\sqrt{f};\nu):= \int \big[\sqrt{f(\eta^{x,x+1})}-\sqrt{f(\eta)}\big]^2 d\nu.$
\end{lemma}
	Note that the condition $h(\eta)=h(\eta^{x,x+1})$ for every $\eta\in\Omega_{nk}$ is crucial here. It means that $h$ depend on the sites $\{x,x+1\}$ only as a function of the sum $\eta(x)+\eta(x+1)$.
In particular,  if $\{x,x+1\}$ is not a slow bond, by definition of $m_n(\cdot)$, and choosing $A$ to be of the form $A=\delta n^2$  the third term on the right hand side vanishes:
\begin{equation}\label{nonslow}
	\int h(w(x+1)-w(x)) f d\nu\;\leq\;\delta n^2 \mf D_{x,x+1}(\sqrt{f},\nu)+\frac{4}{\delta\varepsilon_0 n^2}\int h^2 fd\nu\,.
\end{equation}
On the other hand, if $\{x,x+1\}$ is a slow bond, assuming without loss of generality that $x=ni-1$ for some $0\leq i\leq k-1$, we have that by choosing $A$ of the form $A=\delta n^{2-\beta}$
\begin{equation}\label{slow}
	\begin{split}
		\int h(w(ni)-w(ni-1)) f d\nu\;\leq\; &\delta n^2  n^{-\beta}\mf D_{x,x+1}(\sqrt{f},\nu)\,+\,\frac{4n^{\beta}}{\delta\varepsilon_0 n^2}\int h^2 fd\nu\\
		&- \big[\rho\big(\frac{i}{k}\big)-\rho\big(\frac{i-1}{k}\big)\big]\int hw(in-1)w(in)fd\nu\,.
	\end{split}
\end{equation}

Writing $w((i+1)n-1)-w(in-1)$ as a telescopic sum, by \eqref{nonslow} and \eqref{slow}, and recalling that $\beta>1$, we obtain the following corollary.
\begin{corollary}\label{box}
	Fix $i\in\bb T_k$. Let $h:\Omega_{kn}\to\bb R$ be such that $h(\eta)=h(\eta^{x,x+1})$  for every $in-1 \leq x\leq (1+i)n-2$ and every $\eta\in\Omega_{nk}$. Then, for any $\delta>0$,
	\begin{equation*}
		\begin{split}
			\int h(w((i+1)n-1)-w(in-1)) f d\nu\;\leq\; &\delta n^2 \sum_{x=in-1}^{(i+1)n-1}\xi^{nk}_{x,x+1}\mf D_{x,x+1}(\sqrt{f};\nu)\,+\,\frac{8n^\beta}{\delta\varepsilon_0 n^2}\int h^2 fd\nu\\
			\, - \,&\big[\rho\big(\frac{i}{k}\big)-\rho\big(\frac{i-1}{k}\big)\big]\int hw(in-1)w(in)fd\nu.
		\end{split}
	\end{equation*}
\end{corollary}

\subsection{Estimate of the terms \texorpdfstring{\eqref{longsum_1}}{(4.6)}, \texorpdfstring{\eqref{longsum_2}}{(4.7)} and
\texorpdfstring{\eqref{longsum_3}}{(4.8)}} \label{s43}
The first lemma estimates the sum of \eqref{longsum_1} and \eqref{longsum_2}.
\begin{lemma}\label{p1}
	There exists a constant $C>0$ independent of $n$ and $k$ such that
	\begin{align*}
			&\alpha k^2\int \bigg\{ \sum_{i=0}^{k-1}\sum_{x=in}^{(i+1)n-1}\big\{ w_t(ni)\,-\,w_t(x)\big\}\Big[\rho_t\big(\frac{i-1}{k}\big)\,-\,\rho_t\big(\frac{i}{k}\big)\Big]\\
			&\qquad\qquad+\sum_{i=0}^{k-1}\sum_{x=in}^{(i+1)n-1}\{w_t(n(i+1)-1)\,-\,w_t(x)\}\Big[\rho_t\big(\frac{i+1}{k}\big)\,-\,\rho_t\big(\frac{i}{k}\big)\Big]\bigg\} f_t^n d\nu_t^n\\
			&\leq\; \frac{k^2n^{1+\beta}}{2}\mf D(\sqrt{f_t^n};\nu_t^n)\,+\,Ckn^{2-\beta}.
	\end{align*}
\end{lemma}
\begin{proof}
	The first sum on the left hand side of the inequality in the lemma can be written as
	$$\alpha k^2\int \Big\{ \sum_{i=0}^{k-1}\sum_{x=in}^{(i+1)n-2}\big((i+1)n-1-x\big)\big\{ w_t(x)-w_t(x+1)\big\}\Big[\rho_t\big(\frac{i-1}{k}\big)-\rho_t\big(\frac{i}{k}\big)\Big]\Big\}f_t^nd\nu_t^n.$$
	By Lemma \ref{single} and the Lipschitz continuity of $\rho_t$ for any $\delta>0$, this is bounded by
	\begin{equation*}
		\begin{split}
			Ck\sum_{i=0}^{k-1}\sum_{x=in}^{(i+1)n-2} \Big\{ \delta n^2 \mf D_{x,x+1}(\sqrt{f_t^n};\nu_t^n)+\frac{4}{\delta\varepsilon_0 n^2} n^2 \Big\}
			\;\leq\;&Ck \delta n^2 \mf D(\sqrt{f_t^n})+\frac{4Ck^2n}{\delta\varepsilon_0}\\
			\;=\;&\frac{k^2n^{1+\beta}}{4}\mf D(\sqrt{f_t^n};\nu_t^n)+\frac{Ckn^{2-\beta}}{\varepsilon_0}
		\end{split}
	\end{equation*}
	by choosing $\delta =\frac{kn^{\beta-1}}{4C}$.
	The second sum on the left hand side in the lemma can be dealt with in the same way, concluding the proof.
\end{proof}

The third term to be dealt with, namely \eqref{longsum_3}, is estimated in the next lemma.
\begin{lemma}\label{p2}
	\begin{equation*}
		\begin{split}
			&-\alpha n\int \sum_{i=0}^{k-1}k^2\Big[\rho_t(\frac{i-1}{k})\,-\,\rho_t(\frac{i}{k})\Big]^2w_t(ni-1)w_t(ni) f_t^n d\nu_t^n\\
			&\leq \, \frac{k^2n^{1+\beta}}{2}\mf D(\sqrt{f_t^n};\nu_t^n)\,+\,  C(\varepsilon_0,\kappa)H(\mu_t^n|\nu_t^n)\,+\,o(kn)
		\end{split}
	\end{equation*}
\end{lemma}
\begin{proof}
	Define the function $G:\bb T_{kn}\to\bb R$ by
	\begin{equation*}
		G(x)\;=\;
		\begin{cases}
			-\alpha k^2 \Big[\rho_t(\frac{i-1}{k})\,-\,\rho_t(\frac{i}{k})\Big]^2\,,& \text{ if } x= in-1 \text{ for some } 0\leq i\leq k-1\\
			0\,,& \text{ otherwise. }\\
		\end{cases}
	\end{equation*}
Keep in mind that $G(x)$ is nonzero only when $x$ is the right vertex of a slow bond. Moreover, $\|G\|_\infty$ can be estimated from above by some constant that depends only on $\kappa$ and $\alpha$, and is bounded uniformly over $k$ and $n$ since $\rho_t(\cdot)$ is smooth. With this notation, the term that needs to be estimated in the lemma can be written as
	\begin{equation}\label{intGww}
		\int  n\sum_{x\in\bb T_{kn}}G(x)w_t(x)w_t(x+1) f_t^n d\nu_t^n\,.
	\end{equation}
	Given $\ell\in\bb N$, let $\Lambda_\ell=\{0,1,\ldots,\ell-1\}$ and let $p_{\ell}$ be the uniform measure on $\Lambda_{\ell}$, that is,
	$p_{\ell}(x)=\frac{1}{\ell}$ for $0\leq x\leq \ell-1$ and zero otherwise.
	Let $q_{\ell}$ be the measure on $\bb Z$ given by the convolution of $p_\ell$ with itself, that is,
	$$q_{\ell}(z)\;:=\;\sum_{y\in\bb Z}p_{\ell}(y)p_{\ell}(z-y)\,.$$
	Note that $q_\ell$ has support in $\Lambda_{2\ell-1}$ and that there exists a constant $C>0$ such that $q_\ell(z)\leq\, C\ell^{-1}$ for any $z$.
	Given $\ell<kn/2$, define $$w^{\ell}(x)\;:=\;\sum_{y\in\bb Z}w(x+y)q_{\ell}(y)\,.$$
	The idea to estimate the integral in \eqref{intGww} goes as follows. We will first replace $w_t(x+1)$ by $w_t^{\ell}(x+1)$ and then estimate the remaining integral. Afterwards it only remains to estimate the cost of performing the replacement.

	To simplify the computations, we assume furthermore that $\ell \gg n$ and that $\ell$ is divisible by $n$.
	As can be seen  in~\cite[Lemma 3.2 on p.13 and the comment on p.63, after the proof of Lemma G.2]{jm} that there exists a function $\phi_\ell$ such that for all $x\in \bb T_{kn}$,
	\begin{equation*}
w(x)-w^{\ell}(x)\;=\; \sum_{z\in\bb Z} \phi_{\ell}(z) \big(w(x+z+1)-w(x+z)\big)\,.
	\end{equation*}
	Moreover, the function
	 $\phi_{\ell}(z)$ has support in $\Lambda_{2\ell-1}$, $0\leq \phi_{\ell}(z)\,\leq\, 1$ for every $z\in\bb Z$ and hence $\sum_{z\in\bb Z}\phi_\ell(z)^2\leq 2\ell$.  From the previous identity, rearranging the order of summation,  we obtain
	\begin{equation}\label{Vdif}
		\sum_{x\in\bb T_{kn}}\big( w(x+1)-w^{\ell}(x+1)\big) w(x) G(x)\;=\; \sum_{x\in\bb T_{kn}}h_{x-1}^{\ell}\big(w(x+1)\,-\,w(x)\big)\,,
	\end{equation}
	where, for each $x\in\bb T_{kn}$,
	$$h_x^{\ell}\;:=\;\sum_{z\in\bb Z} \phi_{\ell}(z)w(x-z)G(x-z)\,.$$
	A key observation is that $h_{x-1}^\ell(\eta^{x,x+1})=h_{x-1}^\ell(\eta)$ for every $\eta$, which allows us to apply the integral by parts formula stated in Subsection \ref{s42}.
	Defining
	\begin{align*}
	V(G)&\;:=\;\sum_{x\in\bb T_{kn}}w(x) w(x+1) G(x)\,,\qquad \text{and}\\
	V^{\ell}(G)&\;:=\;\sum_{x\in\bb T_{kn}}w(x) w^{\ell}(x+1) G(x)\,,
	\end{align*}
	the integral in \eqref{intGww} can be rewritten as
	$$n\int V(G)f_t^n d\nu_t^n\,.$$
	As mentioned before, we shall estimate
	\begin{equation*}
	n\int \{V(G)-  V^{\ell}(G)\} f_t^nd\nu_t^n\,, \qquad \text{ and } \qquad n\int V^\ell(G)f_t^n d\nu_t^n\,.
	\end{equation*}
	From the inequalities \eqref{nonslow}, \eqref{slow} and the identity \eqref{Vdif}, we have that for any $\delta>0$,
	\begin{equation}\label{difV}
		\begin{split}
			n\int \{V(G)\,-\,  & V^{\ell}(G)\} f_t^nd\nu_t^n\;\leq\;\delta n^2 \mf D(\sqrt{f_t^n};\nu_t^n)\\
			+\,&\sum_{i=0}^{k-1}\bigg\{\frac{4}{\delta\varepsilon_0 }\sum_{x=in}^{(i+1)n-2}\int (h_{x-1}^{\ell})^2 f_t^nd\nu_t^n\,+\, \frac{4n^\beta}{\delta\varepsilon_0 }\int (h_{(i+1)n-2}^{\ell})^2 f_t^nd\nu_t^n\bigg\}\\
			+&\,\sum_{i=0}^{k-1}\bigg\{ n\Big[\rho_t\big(\frac{i-1}{k}\big)\,-\,\rho_t\big(\frac{i}{k}\big)\Big]\int h_{in-2}^{\ell}w(in-1)w(in)f_t^n d\nu_t^n\bigg\}\,.
		\end{split}
	\end{equation}

	To simplify notation,  define
	\begin{equation}\label{W_ell_G}
	W^{\ell}(G)\,:=\, \sum_{i=1}^k\Big\{n^\beta(h_{in-2}^{\ell})^2\,+\, \sum_{x=(i-1)n}^{in-2}(h_{x-1}^{\ell})^2\Big\}
	\end{equation}
	 and
\begin{equation}\label{Z_ell_G}
Z^{\ell}(G):=\,\sum_{i=0}^{k-1}\Big\{ n\Big[\rho_t(\frac{i-1}{k})\,-\,\rho_t(\frac{i}{k})\Big]\int h_{in-2}^{\ell}w(in-1)w(in)f_t^n d\nu_t^n\Big\}\,.
\end{equation}
	Choosing $\delta=k^2n^{\beta-1}/16$ in \eqref{difV}, we see that in order to prove the theorem, it remains to show that
	\begin{equation}\label{Vl}
		\int nV^{\ell}(G)  f_t^n d\nu_t^n\;\leq\; \frac{k^2n^{\beta+1}}{8}\mf D(\sqrt{f_t^n};\nu_t^n)\,+\,C(\varepsilon_0)\|G\|_\infty \Big\{H(\mu_t^n|\nu_t^n)\,+\, \frac{kn}{\ell} \Big\}\,,
	\end{equation}
	\begin{equation}\label{Wl}
		\frac{4}{k^2n^{\beta-1}\varepsilon_0}\int W^{\ell}(G) f_t^nd\nu_t^n\;\leq\; \frac{C(\varepsilon_0,\kappa)\ell^2}{kn\|G\|_\infty^2}\,,
	\end{equation}
	and
	\begin{equation}\label{Zl}
		\int Z^{\ell}(G)  f_t^n d\nu_t^n\;\leq\;C(\varepsilon_0)\ell\,,
	\end{equation}
	and then choose $\ell=n\sqrt{k}$.
	These estimates will be proven in Lemmas \ref{Vlproof}, \ref{Wlproof} and \ref{Zlproof} respectively.
\end{proof}

\begin{lemma}\label{Vlproof}
	Assume that $n\ll \ell \ll nk$. Then there exists a constant $C(\varepsilon_0)$ such that
	\begin{equation*}
		\int nV^{\ell}(G)  f_t^n d\nu_t^n\;\leq\; \frac{k^2n^{\beta+1}}{8}\mf D(\sqrt{f_t^n};\nu_t^n)\,+\,C(\varepsilon_0)\|G\|_\infty \Big\{H(\mu_t^n|\nu_t^n)\,+\, \frac{kn}{\ell} \Big\}\,.
	\end{equation*}
\end{lemma}
Before coming to the proof we need to bring into play the concept of a subgaussian random variable and the notion of  $\ell$-dependence. Results about these notions are collected in the appendix.
\begin{definition}
	We say that a real-valued random variable $X$ is subgaussian of order $\sigma^2$ if, for every $\theta\in \bb R$,
	$$\log E[e^{\theta X}]\;\leq\;\frac{1}{2}\sigma^2\theta^2.$$
\end{definition}
\begin{definition}\label{def:ldep}
	We say that a set $B\subseteq \T_n$ is $\ell$-sparse if $|i-j|\geq \ell$ for any $i\neq j\in B$. We say that a collection of random variables $\{X_i:i\in\bb T_n\}$ is $\ell$-dependent if for any $\ell$-sparse set $B\subseteq \T_n$ the collection $\{X_i:\, i\in B\}$ is independent.  In particular a collection of $1$-dependent random variables are independent.
\end{definition}
\begin{proof}[Proof of Lemma~\ref{Vlproof}]
	Recall the definition of $V^{\ell}(G)$ and keep in mind that $q_\ell$ can be seen as the convolution of the uniform measure $p_\ell$ with itself. We claim that
\begin{equation}
\label{eq:VlG}
V^{\ell}(G)=\sum_{x\in\bb T_{kn}} \overleftarrow{w}^{\ell}(x)\overrightarrow{w}^{\ell}(x)\,,
\end{equation}
	where
	\begin{align*}
	\overleftarrow{w}^{\ell}(x) &\;:=\; \sum_{y\in\bb Z}w(x-y)G(x-y)p_{\ell}(y)\,,\qquad \text{and}\\
	\overrightarrow{w}^{\ell}(x)&\;:=\; \sum_{y\in\bb Z}w(x+y+1)p_{\ell}(y)\,.
	\end{align*}
	Indeed, note that
	\begin{equation*}
	V^{\ell}(G)= \sum_{x,y,z\in\bb T_{kn}}w(x)G(x)w(x+y+1)p_\ell(z)p_\ell(y-z)\,,
	\end{equation*}
	and
	\begin{equation*}
	\sum_{x\in\bb T_{kn}} \overleftarrow{w}^{\ell}(x)\overrightarrow{w}^{\ell}(x)= \sum_{\tilde x, \tilde y, \tilde z\in\bb T_{kn}}w(\tilde x-\tilde y)G(\tilde x-\tilde y)w(\tilde x+ \tilde z+1)p_\ell(\tilde y)p_\ell(\tilde z)\,.
	\end{equation*}
	Then, to check the claim,  it only remains to use the following change of variables: $x=\tilde x-\tilde y$, $y=\tilde y+\tilde z$, and $z=\tilde y$.

	Recall that $G(z)\neq 0$ if and only if $z=in-1$ for some $0\leq i\leq k-1$. Therefore $\overleftarrow{w}^{\ell}(x)$ can be rewritten as
	$$\overleftarrow{w}^{\ell}(x):=\, \sum_{\topo{i:}{x+2-\ell\leq in\leq x+1}}w(in-1)G(in-1)\ell^{-1}.$$
	Note that since we assume $\ell\gg n$, the cardinality of the  set
	$$S^{\ell,n}(x)\;=\;\Big\{0\leq i\leq k-1: x+2-\ell\leq in\leq x+1\Big\}$$
	grows to infinity as $n\to\infty$. The idea to proceed is to work with a perturbation of $\overleftarrow{w}^{\ell}$. More precisely, it turns out to be useful to replace each appearance of $w(in-1)$  by $\frac{1}{n}\sum_{y=1}^n w(in-y)$ for every $i\in S^{\ell,n}(x)$ in the definition of $\overleftarrow{w}^{\ell}$. More precisely, instead of working with $\overleftarrow{w}^{\ell}$ we will work with $\underleftarrow{w}^{\ell}$, where
	\begin{equation*}
	\underleftarrow{w}^{\ell}(x)\;:=\; \sum_{i\in S^{\ell,n}(x)}\frac{1}{n}\sum_{y=1}^n w(in-y)G(in-1)\ell^{-1}\,.
	\end{equation*}
	Consequently, since our goal is to estimate the expression in~\eqref{eq:VlG} we see that we need to bound
	\begin{equation}\label{eq:Vlrep}
	\int n \sum_{x\in\bb T_{kn}} \big\{\overleftarrow{w}^{\ell}(x)-\underleftarrow{w}^{\ell}(x)\big\} \overrightarrow{w}^{\ell}(x)f_t^n d\nu_t^n\,.
	\end{equation}
	We claim that the above term is bounded from above by
	\begin{equation}\label{Vlrep}
			 \frac{k^2n^{\beta+1}}{8}\mf D(\sqrt{f_t^n};\nu_t^n)\,+\, \frac{C(\varepsilon_0)\|G\|_\infty^2}{k^2n^{\beta-1}}\Big\{H(\mu_t^n|\nu_t^n)\,+\, \frac{kn\log 3}{\ell} \Big\}\,.
		\end{equation}
		Indeed, for each $1\leq y\leq n$, writing $w(in-1)-w(in-y)$ as the telescopic sum
	$$w(in-1)\,-\,w(in-y)\;=\;\sum_{z=1}^{y-1}\{w(in-z)-w(in-z-1)\}\,,$$
	then summing over $y$ from $1$ up to $n$, we get
	$$nw(in-1)\,-\, \sum_{y=1}^n w(in-y)\;=\; \sum_{z=1}^{n-1}(n-z)\Big\{w(in-z)-w(in-z-1)\Big\}.$$
	This identity then implies that
	$$n\big\{\overleftarrow{w}^{\ell}(x)-\underleftarrow{w}^{\ell}(x)\big\}\;=\; \sum_{i\in S^{\ell,n}(x)}G(in-1)\ell^{-1}\sum_{z=1}^{n-1}(n-z)\{w(in-z)-w(in-z-1)\}.$$
	Note that for each chosen $0\leq i\leq k-1$, the number of $x$'s such that $i\in S^{\ell,n}(x)$ is $\ell$. With this fact at hand and inequality  \eqref{nonslow}, we have that
	\begin{equation}\label{overunderdif}
		\begin{split}
			&\int n \sum_{x\in\bb T_{kn}} \big\{\overleftarrow{w}^{\ell}(x)-\underleftarrow{w}^{\ell}(x)\big\} \overrightarrow{w}^{\ell}(x)f_t^n d\nu_t^n\\
			& \leq\; \ell\delta n^2 \mf D(\sqrt{f_t^n};\nu_t^n)\,+\,\frac{4\ell}{\delta \varepsilon_0 n^2} \sum_{x\in\bb T_{kn}} (\|G\|_\infty\ell^{-1}n)^2 \int (\overrightarrow{w}^{\ell}(x))^2 f_t^nd\nu_t^n
		\end{split}
	\end{equation}
	for any $\delta>0$.
To estimate the last term above we need to invoke results about subgaussian random variables which are provided in the appendix.
	By Lemma \ref{Hf} and Lemma \ref{subgsum}, we have
	\begin{equation}\label{avg}
		\overrightarrow{w}^{\ell}(x) \,\, \text{is subgaussian of order} \,\,2(2/\varepsilon_0)^{-2}\ell^{-1}.
	\end{equation}
Moreover, it follows from Definition~\ref{def:ldep} that $\{(\overrightarrow{w}^{\ell}(x))^2: x\in\bb T_{kn}\}$ are $\ell$-dependent.
	In view of Lemma \ref{ldepsum}, for any $\gamma>0$,
	\begin{equation}\label{afterrep_0}
		\int \sum_{x\in\bb T_{kn}} (\overrightarrow{w}^{\ell}(x))^2 f_t^n d\nu_t^n\;\leq\; \frac{2}{\gamma}\Big\{H(\mu_t^n|\nu_t^n)\,+\,\frac{1}{\ell}\sum_{x\in\bb T_{kn}}\log \int e^{\gamma \ell (\overrightarrow{w}^{\ell}(x))^2} d\nu_t^n\Big\}\,.
	\end{equation}
	By \eqref{avg} and Lemma \ref{prod},
	\begin{equation*}
		\int e^{\gamma \ell (\overrightarrow{w}^{\ell}(x))^2} d\nu_t^n \;\leq\; 3
	\end{equation*}
	if $\gamma = C(\varepsilon_0)$. Thus we can conclude that
	\begin{equation*}
		\int \sum_{x\in\bb T_{kn}} (\overrightarrow{w}^{\ell}(x))^2 f_t^n d\nu_t^n\;\leq\; C(\varepsilon_0)\Big\{H(\mu_t^n|\nu_t^n)\,+\, \frac{kn\log 3}{\ell} \Big\}.
	\end{equation*}
	To  finish the proof of the claim, it remains to choose $\delta=k^2n^{\beta-1}\ell^{-1}/8$ in \eqref{overunderdif}.
And to conclude the proof of this lemma, it remains to estimate
	$$\int \sum_{x\in\bb T_{kn}} n\underleftarrow{w}^{\ell}(x)\overrightarrow{w}^{\ell}(x)  f_t^n d\nu_t^n. $$
	Again by Lemma \ref{Hf} and Lemma \ref{subgsum}, recalling that $G(x)\neq 0$ if and only if $x=ni-1$ for some $i\in\bb T_k$, we see that
	\begin{equation}\label{avgG}
		n\underleftarrow{w}^{\ell}(x)\,\, \text{is subgaussian of order} \,\,C(\varepsilon_0)\|G\|_\infty^2\ell^{-1}.
	\end{equation}
	It is easy to see that $\{n\underleftarrow{w}^{\ell}(x)\overrightarrow{w}^{\ell}(x): x\in\bb T_{kn}\}$ are $3\ell-$dependent since $\ell\gg n$.
	In view of Lemma \ref{ldepsum}, for any $\gamma>0$,
	\begin{equation}\label{afterrep}
		\int \sum_{x\in\bb T_{kn}} n\underleftarrow{w}^{\ell}(x)\overrightarrow{w}^{\ell}(x) f_t^n d\nu_t^n\;\leq\; \frac{2}{\gamma}\Big\{H(\mu_t^n|\nu_t^n)\,+\,\frac{1}{3\ell}\sum_{x\in\bb T_{kn}}\log \int e^{3\gamma \ell n  \underleftarrow{w}^{\ell}(x)\overrightarrow{w}^{\ell}(x)} d\nu_t^n\Big\}\,.
	\end{equation}
	By  \eqref{avg}, \eqref{avgG}, and Lemma \ref{prod},
	\begin{equation*}
		\int e^{3\gamma \ell n\underleftarrow{w}^{\ell}(x)\overrightarrow{w}^{\ell}(x)} d\nu_t^n \;\leq\; 3
	\end{equation*}
	if $\gamma = C(\varepsilon_0)\|G\|_\infty^{-1}$ for a constant $C(\varepsilon_0)$, which depends only on $\varepsilon_0$ and is not necessarily equal to the constant $C(\varepsilon_0)$ in \eqref{avgG}. Under this choice of $\gamma$, inequality \eqref{afterrep} gives us that
	$$\int \sum_{x\in\bb T_{kn}} n\underleftarrow{w}^{\ell}(x)\overrightarrow{w}^{\ell}(x)  f_t^n d\nu_t^n\;\leq\; C(\varepsilon_0)\|G\|_\infty \Big\{H(\mu_t^n|\nu_t^n)\,+\, \frac{kn\log 3}{\ell} \Big\}\,.$$
	The Lemma follows immediately from this estimate and the bound of \eqref{eq:Vlrep} by \eqref{Vlrep}.
\end{proof}

Recall the definition of $W^{\ell}(G)$ in \eqref{W_ell_G}.
\begin{lemma}\label{Wlproof}
	There exists a constant $C(\varepsilon_0)>0$ such that
	\begin{equation*}
		\frac{4}{k^2n^{\beta-1}\varepsilon_0}\int W^{\ell}(G) f_t^nd\nu_t^n\;\leq\; \frac{C(\varepsilon_0)\ell^2}{kn\|G\|_\infty^2}\,.
	\end{equation*}
\end{lemma}
\begin{proof}
	Recall that $\phi_\ell(z)\,\leq\, 1$ for every $z\in \bb Z$ and  that $G(x)\neq 0$ if and only if $x=ni-1$ for some $i\in\bb T_k$. Therefore $(h_x^{\ell})^2\,\leq\, C(\varepsilon_0)\|G\|_\infty^2\ell^2 n^{-2}$. Then,
	\begin{equation*}
		\frac{4}{k^2n^{\beta-1}\varepsilon_0}\int W^{\ell}(G) f_t^nd\nu_t^n\;\leq\; \frac{4}{k^2n^{\beta-1}\varepsilon_0}\sum_{i=1}^k C(\varepsilon_0)\|G\|_\infty^2\ell^2 n^{\beta-2}.
	\end{equation*}
	This proves the lemma.
\end{proof}
Recall the definition of $Z^{\ell}(G)$ in \eqref{Z_ell_G}.

\begin{lemma}\label{Zlproof}
	There exists a constant $C(\varepsilon_0,\kappa)>0$ such that
	\begin{equation*}
		\int Z^{\ell}(G)  f_t^n d\nu_t^n\;\leq\; C(\varepsilon_0,k)\ell\,.
	\end{equation*}
\end{lemma}

\begin{proof}
	The proof of this lemma is very similar to the previous lemma. As in the proof of the previous lemma we see that
	$$\Big\lvert  h_{in-2}^{\ell}w(in-1)w(in)\Big\rvert\;\leq\; C(\varepsilon_0)\|G\|_\infty \ell n^{-1}.$$
	From this estimate we can easily deduce that
	\begin{equation*}
		\int Z^{\ell}(G)  f_t^n d\nu_t^n\;\leq\; \sum_{i=1}^k C(\varepsilon_0,\kappa)\ell k^{-1}\,,
	\end{equation*}
	concluding the proof.
\end{proof}

\subsection{Proof of Theorem \ref{infinite}}\label{s44}
In possess of the bound for the relative entropy $H(\mu_t^n|\nu^n_t)$ obtained in Corollary \ref{entropy bound}, we are ready to prove Theorem \ref{infinite}.
\begin{proof}
	Since both $G$ and $\rho(t,\cdot)$ are continuous,
	$$\lim_{k\to\infty}\frac{1}{k}\sum_{i=0}^{k-1}\rho_t(\frac{i}{k})G\big(\frac{i}{k}\big)\;=\;\int_{\bb T} G(u)\rho_t(u)du\,.$$
	Therefore to prove the theorem, it is enough to show that
	\begin{equation}
		\lim_{n\to\infty} E_{\mu_t^n}\Big[\Big|\frac{1}{kn}\sum_{i=0}^{k-1}\Big(\sum_{x=in}^{(i+1)n-1}\big[\eta_t(x)-\rho_t(\frac{i}{k})\big]\Big)G\big(\frac{i}{k}\big)\Big|\Big]\;=\;0\,.
	\end{equation}
	Applying the entropy inequality, for every $\gamma>0$, the expectation in the previous formula is bounded from above by
	\begin{equation}\label{entropy inequality}
		\frac{1}{\gamma kn}\left\{H(\mu_t^n|\nu_t^n)\,+\,\log E_{\nu_t^n}\Big[\exp\Big|\gamma \sum_{i=0}^{k-1}\Big(\sum_{x=in}^{(i+1)n-1}\big[\eta_t(x)-\rho_t(\frac{i}{k})\big]\Big)G\big(\frac{i}{k}\big)\Big|\Big]\right\}.
	\end{equation}
	Since $e^{|x|}\leq e^x+e^{-x}$ and
	$$\limsup_{n}\frac{1}{n}\log(a_n+b_n)\;=\;\max\Big\{\limsup_n \frac{\log a_n}{n}\,,\,\limsup_n\frac{\log b_n}{n}\Big\}\,,$$
	we can remove the absolute value symbol inside the exponential.
	By Lemma \ref{Hf}, since $\eta$ takes value in $[0,1]$ and $\rho_t\in(\varepsilon_0,1-\varepsilon_0)$, under the product measure  $\nu_t^n$, $\eta_t(x)-\rho_t(\frac{i}{k})$ is subgaussian of order $\frac{1}{4}$. By Lemma \ref{subgsum},
	\begin{equation*}
	\sum_{i=0}^{k-1}\bigg(\sum_{x=in}^{(i+1)n-1}\big[\eta_t(x)-\rho_t(\frac{i}{k})\big]\bigg)G\big(\frac{i}{k}\big)
	\end{equation*}
	is subgaussian of order $n\sum_{i=0}^{k-1}\frac{G(i/k)^2}{2}$.
  Therefore, by definition of the subgaussian variable, the expression in \eqref{entropy inequality} is bounded by
	$$\frac{1}{\gamma kn}H(\mu_t^n|\nu^n_t)\,+\,\frac{\gamma n}{kn}\sum_{i=0}^{k-1}\frac{G(i/k)^2}{4}\,.$$
	Choosing $\gamma=\sqrt{\frac{H(\mu_t^n|\nu^n_t)}{kn}}$, we have that
	\begin{equation}\label{repbound}
		E_{\mu_t^n}\Big[\Big|\frac{1}{kn}\sum_{i=0}^{k-1}\Big(\sum_{x=in}^{(i+1)n-1}\big[\eta_t(x)-\rho_t(m_n(x))\big]\Big)G\big(\frac{i}{k}\big)\Big|\Big]\;\leq\; \big(C(G)+1\big)\sqrt{\frac{H(\mu_t^n|\nu^n_t)}{kn}}\,.
	\end{equation}
	To conclude the proof of this theorem, it only remains to use that $H(\mu_t^n|\nu^n_t)=o(kn)$ which was shown in Corollary~\ref{entropy bound}.
\end{proof}

\section{Appendix}

In this appendix we collect some properties about subgaussian random variables. All the results presented here are from Appendix F of \cite{jm}, so we omit their proofs.

The first lemma, known as Hoeffding's Lemma, provides a class of examples of subgaussian random variables.

\begin{lemma}[Hoeffding's lemma, see \cite{Concentration_inequalities}, page 34]\label{Hf}
	Let $X$ be a random variable taking values in $[0,1]$. Then, for any $\theta\in\bb R$,
	$$\log E\big[e^{\theta(X-E[X])}\big]\;\leq\; \frac{1}{8}\theta^2.$$
\end{lemma}

Recall that  $\rho_t(u)\in (\varepsilon_0,1-\varepsilon_0)$ for all $t\in[0,T]$ and all $u\in\bb T$. From this fact and the previous lemma, one can easily deduce that
\begin{equation}
	w(B) \, \text{ is subgaussian of order}\,  (2/\varepsilon_0)^{-2|B|}\,,
\end{equation}
where $w(B):=\prod_{x\in B} w(x)$ and $w(x)$ was defined in \eqref{eq:w}. This is actually also the statement of Lemma F.10 in~\cite{jm}.
The next lemma gives an upper bound for  exponential moments of products of two subgaussian random variables.
\begin{lemma}[Lemma F.8 in \cite{jm}]\label{prod}
	Let $X_i$ be subgaussian random variables of order $\sigma_i^2$, $i=1,2$. Then for any
	$\gamma\leq\,(4\sigma_1\sigma_2)^{-1}$,
	$$E\big[e^{\gamma X_1X_2}\big]\;\leq \;3\,.$$
\end{lemma}
The following lemma is the one-dimensional case of Lemma F.12 in \cite{jm}.

\begin{lemma}\label{subgsum}
	Let $\{X_i:i\in \bb T_n\}$ be $\ell$-dependent. Assume that for any $i\in\bb T_n$, $X_i$ is subgaussian of order $\sigma_i^2$. Then for any $f:\bb T_n\to\bb R$,
	$$\sum_{i\in\bb T_n}f_iX_i \quad \text{is subgaussian of order} \quad 2\ell\sum_{i\in\bb T_n}\sigma_i^2 f_i^2. $$
\end{lemma}

For $\ell$-dependent families $\{X_i: i\in\bb T_n\}$, we have the following form of relative entropy inequality.
\begin{lemma}[Lemma F.4 in \cite{jm}]\label{ldepsum}
	Let $\mu$ be a measure on a finite set $\Omega$ and let $f$ be a density with respect to $\mu$.  Let $\{X_i: i\in\bb T_k\}$ be $\ell$-dependent with respect to $\mu$ with $\ell<n/2$. Then for any $\gamma>0$,
	$$\int \sum_{i\in\bb T_k} X_i f d\mu\;\leq\; \frac{2}{\gamma}\Big\{H(f|\mu)\,+\,\frac{1}{\ell}\sum_{i\in\bb T_k}\log \int e^{\gamma \ell X_i} d\mu\Big\}\,,$$
	where $H(f |\mu)$ is the relative entropy of $d\nu =f d\mu$ with respect to $d\mu$.

\end{lemma}

\section*{Acknowledgements}
D.E.\ was supported by the National Council for Scientific and Technological Development (CNPq) via a Bolsa de Produtividade  303348/2022-4. D.E.\ moreover acknowledge support by the Serrapilheira Institute (Grant Number Serra-R-2011-37582). D.E.\ and T.F.\  acknowledge support by the CNPq via a Universal Grant (Grant Number 406001/2021-9). T.F.\ was supported by the  CNPq via a Bolsa de Produtividade number 311894/2021-6. D.E. and T.F  were partially
supported by FAPESB (Edital FAPESB Nº 012/2022 - Universal - NºAPP0044/2023). T.X. was supported by Serrapilheira Institute through a postdoc scholarship (Grant Number Serra-R-2011-37582).

\bibliographystyle{plain}
\bibliography{bibliography}

\end{document}